\documentclass[11pt,a4paper]{article}
\usepackage{a4wide}
\usepackage{tikz}
\usetikzlibrary{decorations.pathmorphing}

\usepackage{amssymb} 
\usepackage{amsmath} 
\usepackage{amsthm} 
\usepackage{booktabs} 

\newtheorem{theorem}{Theorem}
\newtheorem{proposition}[theorem]{Proposition}
\newtheorem{lemma}[theorem]{Lemma}
\newtheorem{conjecture}[theorem]{Conjecture}
\newtheorem{observation}[theorem]{Observation}

\newtheorem{claim}{Claim}
\newtheorem{question}[theorem]{Question}

\newcommand{\eps}{\varepsilon}
\newcommand{\Goldberg}{Gol\kern-0.35mm\char39\kern-0.35mm dberg}

\DeclareMathOperator{\ch}{ch}

\newcommand{\MK}{\mathcal{K}}

\newcommand{\qitem}[1]{\noindent\leavevmode\hangindent1.5\parindent%
       \noindent\hbox to1.5\parindent{#1\hss}\ignorespaces}

\tikzstyle{vertex}=[inner sep = 0pt, minimum width=4pt, fill=black,
shape=circle]
\tikzstyle{xblue}=[dotted, very thick]
\tikzstyle{xred}=[decorate, decoration={snake, segment length=.2cm,
  amplitude=.05cm}]
\newcommand{\gpoint}[2]{\node[style=vertex, label=#1:$#2$]}
\newcommand{\bpoint}[1]{\gpoint{below}{#1}}
\newcommand{\apoint}[1]{\gpoint{above}{#1}}

\title{Extension from Precoloured Sets of Edges}

\author{Katherine Edwards\,\thanks{\,Mathematics of Networks Department,
    Nokia Bell Labs, Murray Hill, NJ, USA;
    \texttt{katherine.edwards2@gmail.com}.}
  \and Ant\'onio Gir\~ao\,\thanks{\,Department of Pure Mathematics and
    Mathematical Statistics, University of Cambridge, Cambridge, UK;
    \texttt{A.Girao@dpmms.cam.ac.uk}.}
  \and Jan van den Heuvel\,\thanks{\,Department of Mathematics, London
    School of Economics and Political Science, London, UK;
    \texttt{j.van-den-heuvel@lse.ac.uk}.}
  \and Ross J. Kang\,\thanks{\,Department of Mathematics, Radboud
    University, Nijmegen, Netherlands;
    \texttt{ross.kang@gmail.com}. Supported by Veni (639.031.138) and Vidi (639.032.614)
    grants of the Netherlands Organisation for Scientific Research (NWO).}
  \and Gregory J. Puleo\,\thanks{\,Department of Mathematics and Statistics, Auburn University, Auburn, AL, USA;
    \texttt{gjp0007@auburn.edu}.}
  \and Jean-S\'ebastien Sereni\,\thanks{\,CNRS, CSTB (ICube),
    Strasbourg, France; \texttt{sereni@kam.mff.cuni.cz}.
    This author's work was partially supported by \emph{Agence Nationale de
      la Recherche} under references \textsc{anr 10 jcjc 0204 01} and
    \textsc{anr 13 BS02 0007}.}}

\begin{document}

\maketitle
\begin{abstract}
  \noindent
  We consider precolouring extension problems for proper edge-colourings of
  graphs and multigraphs, in an attempt to prove stronger versions of
  Vizing's and Shannon's bounds on the chromatic index of (multi)graphs in
  terms of their maximum degree~$\Delta$. We are especially interested in
  the following question: when is it possible to extend a precoloured
  matching to a colouring of all edges of a (multi)graph? This question
  turns out to be related to the notorious List Colouring Conjecture
  and other classic notions of choosability.
\end{abstract}

\section{Introduction}

Let $G = (V,E)$ be a (multi)graph and let $\MK=[K]=\{1,\dots,K\}$ be a
palette of available colours. (In this paper, a \emph{multigraph} can have
multiple edges, but no loops; while a \emph{graph} is always simple.) We
consider the following question: given a subset $S\subseteq E$ of edges and
a proper colouring of elements of~$S$ (i.e., adjacent edges must receive
distinct colours) using only colours from~$\MK$, is there a proper
colouring of all edges of~$G$ (again using only colours from~$\MK$) in
concordance with the given colouring on~$S$? We may consider the set~$S$ as
a set of \emph{precoloured} edges, while the full colouring, if it exists,
may be considered as \emph{extending} the precolouring. If the set~$S$
forms a matching in~$G$, then the precolouring of~$S$ may be arbitrary
from~$\MK$.

An early appearance of a problem regarding precolouring extension of
edge-colourings can be found in Marcotte and Seymour~\cite{MS90}.
Note also that the completion of partial Latin squares can be interpreted as an edge-precolouring extension problem restricted to complete bipartite graphs, and this has been studied since as early as 1960, cf.~e.g.,~\cite{Sme81}.
Nevertheless, in general the edge-precolouring extension problem has been less comprehensively studied than its vertex-colouring counterpart. 
We hope to provoke interest in
edge-precolouring extension and in 
the following question especially.

\begin{question}\label{question:main}\mbox{}\\*
  Let~$G$ be a multigraph with maximum degree $\Delta(G)$ and maximum
  multiplicity $\mu(G)$, let $\MK=[\Delta(G)+\mu(G)]$, and let~$M$ be a
  matching of~$G$ precoloured from the palette~$\MK$. What conditions
  on~$G$ and~$M$ ensure that the precolouring of $M$ extends to a proper
  $\MK$-edge-colouring of all of~$G$?
\end{question}

\noindent
The obvious relationship between edge-precolouring and its vertex
counterpart --- in which we can see edge-precolouring extension of~$G$ as
vertex-precolouring extension in its line graph $L(G)$ --- yields immediate
implications. For us, the distance between two edges in~$G$ is their
corresponding distance in $L(G)$, i.e., the number of \emph{vertices}
contained in a shortest path in~$G$ between any of their end-vertices. A
\emph{distance-$t$ matching} is a set of edges having pairwise distance
greater than $t$. (This means that a matching is a distance-$1$ matching,
while an induced matching is a distance-$2$ matching. Any set of edges is a
distance-$0$ matching.) We point out the following consequence of a result
of Albertson \cite[Thm.~4]{Alb98} (see Subsection~\ref{sub:background}) and
Vizing's theorem.

\begin{proposition}\label{prop:albertson1}\mbox{}\\*
  Let~$G$ be a multigraph with maximum degree $\Delta(G)$ and maximum
  multiplicity $\mu(G)$. Using the palette $\MK=[\Delta(G)+\mu(G)+1]$, any
  precoloured distance-$3$ matching can be extended to a proper
  edge-colouring of all of~$G$.
\end{proposition}

\noindent
Albertson and Moore~\cite[Conj.~1]{AlMo01} conjectured that when $G$ is a
simple graph, any precoloured distance-$3$ matching can be extended to a
proper edge-colouring of $G$ using the palette $\MK=[\Delta(G)+1]$. We
propose a stronger conjecture.

\begin{conjecture}\label{conj:main}\mbox{}\\*
  Let~$G$ be a multigraph with maximum degree $\Delta(G)$ and maximum
  multiplicity $\mu(G)$. Using the palette $\MK=[\Delta(G)+\mu(G)]$, any
  precoloured distance-$2$ matching can be extended to a proper
  edge-colouring of all of~$G$.
\end{conjecture}

\noindent
Conjecture~\ref{conj:main} strengthens Proposition~\ref{prop:albertson1} in
two ways: we impose a weaker constraint on the distance between precoloured
edges, and we use a smaller palette. Evidently, we believe that in
edge-precolouring the distance requirement ought to be not as strong as it
is for vertex-precolouring extension. In Section~\ref{sec:counterexample},
however, we show how Conjecture~\ref{conj:main} becomes false if we are
allowed to precolour a distance-$1$ rather than a distance-$2$ matching. 
Note that Conjecture~\ref{conj:main} easily becomes false, even for trees, if we replace the
palette~$\MK$ by $[\Delta(G)]$ or by $[\chi'(G)]$, where $\chi'(G)$ is the
chromatic index of~$G$. For
instance, consider stars with each edge subdivided exactly once; see
Figure~\ref{fig:extree}.

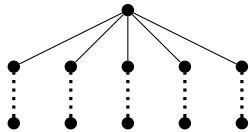
\begin{figure}
  \centering

\begin{tikzpicture}[-,>=,node distance=0.75cm,scale=1,draw,nodes={circle,draw,fill=black, inner sep=1.5pt}]

  \node (i1) {};
  \node (i2) [right of=i1] {};
  \node (i3) [right of=i2] {};
  \node (i4) [right of=i3] {};
  \node (i5) [right of=i4] {};

  \node (0) [above of=i3] {};

  \node (j1) [below of=i1] {};
  \node (j2) [below of=i2] {};
  \node (j3) [below of=i3] {};
  \node (j4) [below of=i4] {};
  \node (j5) [below of=i5] {};

  \path (0) edge [] (i1);
  \path (0) edge [] (i2);
  \path (0) edge [] (i3);
  \path (0) edge [] (i4);
  \path (0) edge [] (i5);

  \path (j1) edge [dotted,very thick,left] (i1);
  \path (j2) edge [dotted,very thick,left] (i2);
  \path (j3) edge [dotted,very thick,left] (i3);
  \path (j4) edge [dotted,very thick,left] (i4);
  \path (j5) edge [dotted,very thick,left] (i5);

%

\end{tikzpicture}
  \caption{A representative $G$ of a class of trees, with a non-extendable
    precoloured \mbox{(distance-$2$)} matching, using the palette
    $[\Delta(G)]=[\chi'(G)]$. Dashed lines indicate edges precoloured with
    colour~$1$.\label{fig:extree}}
\end{figure}

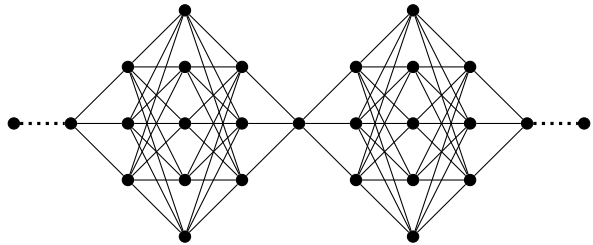
\begin{figure}
  \centering

\begin{tikzpicture}[-,>=,node distance=0.75cm,scale=1,draw,nodes={circle,draw,fill=black, inner sep=1.5pt}]

\node (a1) {};

\node (b1) [right of=a1] {};

  \path (a1) edge [dotted,very thick] (b1);

\node (c2) [right of=b1] {};
\node (c1) [above of=c2] {};
\node (c3) [below of=c2] {};

  \path (b1) edge [] (c1);
  \path (b1) edge [] (c2);
  \path (b1) edge [] (c3);

\node (d3) [right of=c2] {};
\node (d2) [above of=d3] {};
\node (d1) [above of=d2] {};
\node (d4) [below of=d3] {};
\node (d5) [below of=d4] {};

  \path (c1) edge [] (d1);
  \path (c1) edge [] (d2);
  \path (c1) edge [] (d3);
  \path (c1) edge [] (d4);
  \path (c1) edge [] (d5);
  \path (c2) edge [] (d1);
  \path (c2) edge [] (d2);
  \path (c2) edge [] (d3);
  \path (c2) edge [] (d4);
  \path (c2) edge [] (d5);
  \path (c3) edge [] (d1);
  \path (c3) edge [] (d2);
  \path (c3) edge [] (d3);
  \path (c3) edge [] (d4);
  \path (c3) edge [] (d5);

\node (e2) [right of=d3] {};
\node (e1) [above of=e2] {};
\node (e3) [below of=e2] {};

  \path (e1) edge [] (d1);
  \path (e1) edge [] (d2);
  \path (e1) edge [] (d3);
  \path (e1) edge [] (d4);
  \path (e1) edge [] (d5);
  \path (e2) edge [] (d1);
  \path (e2) edge [] (d2);
  \path (e2) edge [] (d3);
  \path (e2) edge [] (d4);
  \path (e2) edge [] (d5);
  \path (e3) edge [] (d1);
  \path (e3) edge [] (d2);
  \path (e3) edge [] (d3);
  \path (e3) edge [] (d4);
  \path (e3) edge [] (d5);

\node (f1) [right of=e2] {};

  \path (f1) edge [] (e1);
  \path (f1) edge [] (e2);
  \path (f1) edge [] (e3);

\node (g2) [right of=f1] {};
\node (g1) [above of=g2] {};
\node (g3) [below of=g2] {};

  \path (f1) edge [] (g1);
  \path (f1) edge [] (g2);
  \path (f1) edge [] (g3);

\node (h3) [right of=g2] {};
\node (h2) [above of=h3] {};
\node (h1) [above of=h2] {};
\node (h4) [below of=h3] {};
\node (h5) [below of=h4] {};

  \path (g1) edge [] (h1);
  \path (g1) edge [] (h2);
  \path (g1) edge [] (h3);
  \path (g1) edge [] (h4);
  \path (g1) edge [] (h5);
  \path (g2) edge [] (h1);
  \path (g2) edge [] (h2);
  \path (g2) edge [] (h3);
  \path (g2) edge [] (h4);
  \path (g2) edge [] (h5);
  \path (g3) edge [] (h1);
  \path (g3) edge [] (h2);
  \path (g3) edge [] (h3);
  \path (g3) edge [] (h4);
  \path (g3) edge [] (h5);

\node (i2) [right of=h3] {};
\node (i1) [above of=i2] {};
\node (i3) [below of=i2] {};

  \path (i1) edge [] (h1);
  \path (i1) edge [] (h2);
  \path (i1) edge [] (h3);
  \path (i1) edge [] (h4);
  \path (i1) edge [] (h5);
  \path (i2) edge [] (h1);
  \path (i2) edge [] (h2);
  \path (i2) edge [] (h3);
  \path (i2) edge [] (h4);
  \path (i2) edge [] (h5);
  \path (i3) edge [] (h1);
  \path (i3) edge [] (h2);
  \path (i3) edge [] (h3);
  \path (i3) edge [] (h4);
  \path (i3) edge [] (h5);

\node (j1) [right of=i2] {};

  \path (j1) edge [] (i1);
  \path (j1) edge [] (i2);
  \path (j1) edge [] (i3);

\node (k1) [right of=j1] {};

  \path (j1) edge [dotted,very thick] (k1);

\end{tikzpicture}
  \caption{A representative~$G$ of a class of bipartite graphs, with a
    non-extendable matching consisting of two edges, using the palette
    $[\Delta(G)]=[\chi'(G)]$. Dashed lines indicate edges precoloured with
    colour $1$.\label{fig:exdistant}}
\end{figure}

In
another direction, one might wonder if a strong enough distance
requirement on the precoloured matching permits us to take a smaller palette,
like $[\Delta(G)]$ or $[\chi'(G)]$. This fails however, even for bipartite
graphs, as we now show.

First, for any positive integer $m$, let $D_m$ denote the bipartite graph
on vertex set $\{x\}\cup A^x\cup B\cup A^y\cup\{y\}$, where $|A^x|=|A^y|=m$
and $|B|=2m-1$, and whose edge set is the set of all pairs between
$\{x\}\cup B$ and $A^x$ and between $\{y\}\cup B$ and $A^y$. Let us observe
an easy property of the graph $D_m$: in any proper edge-colouring of $D_m$
with colours from $[2m]$, there must be at least one edge of colour $1$
incident to $x$ or $y$. For otherwise, since each vertex in $A^x$ has
degree $2m$, there must be $m$ edges of colour $1$ between $A^x$ and $B$;
similarly, there must be $m$ edges of colour $1$ between $A_y$ and $B$. But
this implies that there are $2m$ distinct edges of colour $1$ incident to
the $2m-1$ vertices in $B$, which means that a vertex of $B$ is incident to
two edges of colour $1$, a contradiction.

Next, for any positive integers $\ell,m$, let $G_{m,\ell}$ be the graph
formed by taking $\ell$ disjoint copies $H_1,\dotsc,H_\ell$ of $D_m$ with
vertex sets labelled $\{x_i\}\cup A^x_i\cup B_i\cup A^y_i\cup\{y_i\}$,
identifying $y_i$ with $x_{i+1}$ for all $i=1,\dotsc,\ell-1$, and then
adding two new vertices $x'$ and $y'$ and two new edges~$x'x_1$
and~$y_\ell y'$. See Figure~\ref{fig:exdistant} for a depiction of
$G_{3,2}$. It is straightforward to check that~$G_{m,\ell}$ is bipartite,
has maximum degree $2m$, and that the edges $x'x_1$ and $y_\ell y'$ are at
distance $4\ell+1$ in $G_{m,\ell}$. Consider a precolouring of $G_{m,\ell}$
from the palette $[2m]=[\Delta(G_{m,\ell})]=[\chi'(G_{m,\ell})]$ in which
the edges $x'x_1$ and $y_\ell y'$ are precoloured $1$. Suppose, for a
contradiction, that there is a proper extension of this precolouring. Then
there can be no edge of colour $1$ between~$A^y_1$ and~$B_1$. By our
observation about $D_m$, there must be an edge of colour $1$ between
$A^y_1$ and $y_1=x_2$. It follows by an induction (via copies of $D_m$)
that there is an edge of colour $1$ between $A^y_\ell$ and $y_\ell$. Since
$y_\ell y'$ is precoloured $1$, we have arrived at our desired
contradiction.

If true, Conjecture~\ref{conj:main} would extend Vizing's
theorem~\cite{Viz64}, which is independently due to Gupta,
cf.~\cite{Gup74}. A variant of Conjecture~\ref{conj:main} was proved by
Berge and Fournier~\cite[Cor.~2]{BeFo91} --- they showed that extension is
guaranteed, even from precoloured distance-$1$ matchings, provided that all
edges of the matching have been precoloured with the same colour.

In this paper, we prove several special cases of
Conjecture~\ref{conj:main}, in particular, for bipartite multigraphs,
subcubic multigraphs, and planar graphs of large enough maximum degree.
Indeed, for these classes we show that Conjecture~\ref{conj:main} holds
even when the precoloured set is allowed to be a distance-$1$ matching.
Moreover, we prove a variant of Conjecture~\ref{conj:main}, where the
extended edge-colouring \emph{avoids} some prescribed colours on a
(distance-$1$) matching. We discuss this further in
Subsection~\ref{sub:main}. However, first allow us to place the conjecture
in context by giving some preliminary observations.

By the following easy observation, Conjecture~\ref{conj:main} is also
related to list edge-colouring, and therefore to the \emph{List Colouring
  Conjecture} (LCC), which states that $\ch'(G)=\chi'(G)$ for any
multigraph~$G$ (where $\ch'(G)$ is, as usual, the list chromatic index
of~$G$).

For a non-precoloured edge, we define its \emph{precoloured degree} as the
number of adjacent precoloured edges.

\begin{observation}\label{cor:choose}\mbox{}\\*
  Let~$G$ be a multigraph with list chromatic index $\ch'(G)$. For a
  positive integer~$k$, take the palette as $\MK=[\ch'(G)+k]$. If~$G$ is
  properly precoloured so that the precoloured degree of any
  non-precoloured edge is at most~$k$, then the precolouring can be
  extended to a proper edge-colouring of all of~$G$.  
\end{observation}

\noindent
So, if we assume that the LCC holds, then the following weak form of
Conjecture~\ref{conj:main} holds as well: using the palette
$\MK=[\Delta(G)+\mu(G)+1]$, any precoloured distance-$2$ matching extends
to all of~$G$. Observation~\ref{cor:choose} follows from a more refined
statement we will give in Section~\ref{sec:bipartite}.

Due to the remarkable work of Kahn~\cite{Kah96b,Kah96a,Kah00} on
edge-colourings and list edge-colourings of (multi)graphs, not only does an
asymptotic form of Conjecture~\ref{conj:main} hold, but so does a
precolouring extension of an asymptotic form of the Goldberg--Seymour
Conjecture (which we review in Subsection~\ref{sub:background}). Kahn's
theorem and Observation~\ref{cor:choose} together imply the following.

\begin{proposition}\label{prop:kahn}\mbox{}\\*
  For any $\eps>0$, there exists a constant~$C_\eps$ such that the
  following holds. For any multigraph~$G$ with $\chi'(G)\ge C_\eps$, any
  precoloured matching using the palette $\MK=[(1+\eps)\chi'(G)]$ can be
  extended to a proper edge-colouring of all of~$G$.
\end{proposition}

\noindent
If we replace $\chi'(G)$ in the statement above by $\Delta(G)+\mu(G)$ or by
the Goldberg--Seymour bound, then the statement remains valid, either due
to Vizing's theorem or due to another theorem of Kahn.

One of our motivations for the formulation and study of
Conjecture~\ref{conj:main} comes from the close connections with
vertex-precolouring and with the LCC.

\subsection{Main Results}\label{sub:main}

Although it appears that the LCC and our conjecture are independent
statements, we have obtained several results corresponding to specific
areas of success in list edge-colouring. In summary, we confirm
Conjecture~\ref{conj:main} for bipartite multigraphs, subcubic multigraphs,
and planar graphs of large enough maximum degree. We also obtain a
precolouring extension variant of Shannon's theorem, and we confirm a
relaxed version of Conjecture~\ref{conj:main}, where the extended
edge-colouring \emph{avoids} some prescribed colours on a matching.
Furthermore, all of these partial results hold in the more general context
where the precoloured set is allowed to be a distance-$1$ matching, rather
than the distance-$2$ matching required by Conjecture~\ref{conj:main}. In
fact, in this section we mostly present our main results restricted to
precoloured matchings, to aid clarity, even when yet more general
statements hold.

Our first result is an edge-precolouring extension of K\H{o}nig's theorem
that any bipartite multigraph~$G$ is $\Delta(G)$-edge-colourable, whereas
the subsequent result is an edge-precolouring analogue of Shannon's theorem
that any multigraph~$G$ is
$\bigl\lfloor\tfrac32\Delta(G)\bigr\rfloor$-edge-colourable.

\begin{theorem}\label{thm:bipartite}\mbox{}\\*
  Let~$G$ be a bipartite multigraph with maximum degree $\Delta(G)$. With
  the palette $\MK=[\Delta(G)+1]$, any precoloured matching can be extended
  to a proper edge-colouring of all of~$G$.
\end{theorem}

\noindent
As indicated in Figures~\ref{fig:extree} and~\ref{fig:exdistant}, the palette
size in Theorem~\ref{thm:bipartite} is sharp.

\begin{theorem}\label{thm:shannonmain}\mbox{}\\*
  Let~$G$ be a multigraph with maximum degree $\Delta(G)$. With the palette
  $\MK=\bigl[\,\bigl\lfloor\tfrac32\Delta(G)+\tfrac12\bigr\rfloor\,\bigr]$,
  any precoloured matching can be extended to a proper edge-colouring of
  all of~$G$.
\end{theorem}

\noindent
Due to the Shannon multigraphs, this last statement is sharp if $\Delta(G)$
is even, and within~$1$ of being sharp if $\Delta(G)$ is odd.
Theorems~\ref{thm:bipartite} and~\ref{thm:shannonmain} are proved in
Section~\ref{sec:bipartite} using powerful list colouring tools developed
by Borodin, Kostochka and Woodall~\cite{BKW97}.

\medskip
The following theorem concerns multigraphs that are subcubic, i.e., of
maximum degree at most~$3$. Note that Theorem~\ref{thm:subcubic} improves
upon Theorem~\ref{thm:shannonmain} for $\Delta(G)=3$.

\begin{theorem}\label{thm:subcubic}\mbox{}\\*
  Let~$G$ be a subcubic multigraph. With the palette $\MK=[4]$, any
  precoloured matching can be extended to a proper edge-colouring of all
  of~$G$.
\end{theorem}

\noindent
The example we give in Section~\ref{sec:counterexample} shows that~$3$ is
the largest value of $\Delta(G)$ for which we are guaranteed that the
palette $[\Delta(G)+1]$ is enough to extend every precoloured matching to a
proper edge-colouring of the whole graph. In other words,
Theorem~\ref{thm:subcubic} is best possible with respect to $\Delta(G)$. A
form of Theorem~\ref{thm:subcubic}, for subcubic simple graphs and with a
distance condition on the precoloured matching, was observed by Albertson
and Moore~\cite{AlMo01}. Although the LCC remains open for subcubic graphs,
Juvan, Mohar and \v{S}krekovski~\cite{JMS98} have made a significant
attempt. They showed that for any subcubic graph~$G$, if lists of~$3$
colours are given to the edges of a subgraph~$H$ with $\Delta(H)\le2$ and
lists of~$4$ colours to the other edges, then~$G$ has a proper
edge-colouring using colours from those lists.

Theorem~\ref{thm:subcubic} is a direct corollary of the following theorem,
which may be of interest in its own right. Its proof uses a
degree-choosability condition and can be found in Section~\ref{sec:gallai}.

\begin{theorem}\label{thm:gallai}\mbox{}\\*
  Let~$G$ be a connected multigraph with maximum degree $\Delta(G)$. Choose
  a non-negative integer~$k$ such that $\Delta(L(G))\le\Delta(G)+k$, and
  take the palette as $\MK=[\Delta(G)+k]$. If~$G$ is properly precoloured
  so that the precoloured degree of any vertex is at most~$k$, then the
  precolouring can be extended to a proper edge-colouring of all of~$G$,
  except in the following cases:

  \smallskip
  \qitem{(a)}$k=0$ and~$G$ is a simple odd cycle;

  \smallskip
  \qitem{(b)}$G$ is a triangle with edges of multiplicity $m_1,m_2,m_3$ and
  $k=\min\{m_1,m_2,m_3\}-1$.
\end{theorem}

\noindent
Note that, when restricted to precoloured matchings, this theorem produces
weak or limited bounds for larger maximum degree. On the other hand, if we
replace every precoloured edge in the example of Figure~\ref{fig:extree} by
a precoloured multi-edge of multiplicity~$k$ (or $k+1$) and a precolouring
from~$[k]$ (or $[k+1]$), we see that the palette bound (or precoloured
degree condition) is best possible.

In the case where $k=\Delta(L(G))-\Delta(G)$ in Theorem~\ref{thm:gallai},
the number of colours used is equal to the maximum degree of the line
graph. In that sense the theorem can be considered as a precolouring
extension of Brooks's theorem restricted to line graphs. It is relevant to
mention that vertex-precolouring extension versions of Brooks's
theorem~\cite{AKW05,Axe03,Rac09} require, among other conditions, a large
minimum distance between the precoloured vertices.

\medskip
The class of planar graphs could be of particular interest. There is a
prominent line of work on (list) edge-colouring for this class, which we
discuss further in Subsection~\ref{sub:background} and
Section~\ref{sec:planar}. Our main contributions to this area are the
following results, the second one of which can be viewed as a strengthening
of another old result of Vizing~\cite{Viz65}, provided the graph's maximum
degree is large enough.

\begin{theorem}\label{thm:planar}\mbox{}\\*
  Let~$G$ be a planar graph with maximum degree $\Delta(G)\ge17$. Using the
  palette $\MK=[\Delta(G)+1]$, any precoloured matching can be extended to
  a proper edge-colouring of all of~$G$.
\end{theorem}

\begin{theorem}\label{thm:planar2}\mbox{}\\*
  Let~$G$ be a planar graph with maximum degree $\Delta(G)\ge23$. Using the
  palette $\MK=[\Delta(G)]$, any precoloured distance-$3$ matching can be
  extended to a proper edge-colouring of all of~$G$.
\end{theorem}

\noindent
Due to the trees exhibited in Figure~\ref{fig:extree}, the palette size in
Theorem~\ref{thm:planar} cannot be reduced, while the minimum distance
condition in Theorem~\ref{thm:planar2} cannot be weakened. In
Section~\ref{sec:planar}, we give some more results on when it is possible
for a precoloured matching in a planar graph to be extended. A summary of
the results is given in Table~\ref{tab:planar}.

\begin{table}
  \centering
  \begin{tabular}{clclp{42mm}}
    \toprule
    & Palette $\MK$ & Distance $t$ & Max.\ degree $\Delta$
    \qquad\qquad\quad & Reference\\
    \midrule
    1. & $[\Delta+4]$ & 1 & all $\Delta$ & Thm.~\ref{thm:shannonmain}
    ($\Delta\le7$) and\hspace*{\fill}\linebreak
    Obs.~\ref{cor:choose} with~\cite{Bon15} ($\Delta\ge8$)\\
    2. & $[\Delta+3]$ & 1 & $\Delta\le5$; $\Delta\ge8$ &
    Thm.~\ref{thm:shannonmain}; Obs.~\ref{cor:choose} with~\cite{Bon15}\\
    3. & $[\Delta+2]$ & 1 & $\Delta\ge12$ & Obs.~\ref{cor:choose}
    with~\cite{BKW97}\\
    4. & $[\Delta+2]$ & 2 & $\Delta\le4$; $\Delta\ge8$ & \cite{JMS99};
    Obs.~\ref{cor:choose} with~\cite{Bon15}\\
    5. & $[\Delta+2]$ & 3 & all $\Delta$ & Prop.~\ref{prop:albertson1}\\
    6. & $[\Delta+1]$ & 1 & $\Delta\le3$; $\Delta\ge17$ &
    Thm.~\ref{thm:subcubic}; Thm.~\ref{thm:planar}\\
    7. & \mbox{}$[\Delta+1]$ & \mbox{}2 & $\Delta\ge12$ &
    Obs.~\ref{cor:choose} with~\cite{BKW97}\\
    8. & $[\Delta]$ & 3 & $\Delta\ge23$ & Thm.~\ref{thm:planar2}\\
    \bottomrule
  \end{tabular}
  \caption{Summary of edge-precolouring extension results for planar graphs
    with maximum degree~$\Delta$, when a distance-$t$ matching~$M$ is
    precoloured using the palette~$\MK$. See Section~\ref{sec:planar} for
    further details how these results can be obtained.}\label{tab:planar}
\end{table}

\medskip
Suppose that we would go to any means to obtain an extension form of
Vizing's theorem, say, by weakening the precolouring condition. We still
let $\MK=[K]$ be a palette of available colours. Given a subset
$S\subseteq E$ of edges and an arbitrary (i.e., not necessarily proper)
colouring of elements of~$S$ using only colours from~$\MK$, is there a
proper colouring of all edges of~$G$ (using colours from~$\MK$) that
\emph{differs from} the given colouring on every edge of~$S$? We may
consider the coloured set~$S$ as a set of \emph{forbidden} (coloured)
edges, while the full colouring, if it can be produced, is called an
\emph{avoidance} of the forbidden edges. We can show the following result,
which, while it is in one sense weaker than the statement in
Conjecture~\ref{conj:main}, is directly implied by neither the LCC nor
other existing precolouring results, implies Vizing's theorem, and provides
further evidence in support of Conjecture~\ref{conj:main}. (This result was
stated as a conjecture in an earlier version of this paper.)

\begin{theorem}\label{thm:forbidden}\mbox{}\\*
  Let~$G$ be a multigraph with maximum degree $\Delta(G)$ and maximum
  multiplicity $\mu(G)$. Using the palette $\MK=[\Delta(G)+\mu(G)]$, any
  forbidden matching can be avoided by a proper edge-colouring of all
  of~$G$.
\end{theorem}

\noindent
We use an aforementioned result of Berge and Fournier and a recolouring
argument to prove this theorem in Section~\ref{sec:forbidden}.

\medskip
Some basic knowledge of edge-colouring is a prerequisite to the
consideration of edge-precolouring extension problems --- we provide some
related background in the next subsection. To our frustration, many of the
major methods for colouring edges (such as Kempe chains, Vizing fans,
Kierstead paths, Tashkinov trees) seem to be rendered useless by
precoloured edges. Though Conjecture~\ref{conj:main} may at first seem as
if it should be an ``easy extension'' of Vizing's theorem, it might well be
difficult to confirm (if true). We are keen to learn of related
edge-precolouring results independent of current list colouring
methodology.

\subsection{Further Background}\label{sub:background}

Edge-colouring is a classic area of graph theory. We give a quick overview
of some of the most relevant history for our study. The reader is referred
to the recent book by Stiebitz, Scheide, Toft and Favrholdt~\cite{SSTF12}
for detailed references and fuller insights. The lower bound
$\chi'(G)\ge\Delta(G)$ is obviously true for any multigraph $G$. Close to a
century ago, K\H{o}nig proved that all bipartite multigraphs meet this
lower bound with equality. Shannon~\cite{Sha49} in 1949 proved that
$\chi'(G)\le\bigl\lfloor\tfrac32\Delta(G)\bigr\rfloor$ for any multigraph
$G$. Somewhat later, Gupta (as mentioned in~\cite{Gup74}) and,
independently, Vizing~\cite{Viz64} proved that
$\chi'(G)\le\Delta(G)+\mu(G)$ for any multigraph~$G$, so
$\chi'(G)\in\{\Delta(G),\Delta(G)+1\}$ if~$G$ is simple. Both the Shannon
bound and the Gupta--Vizing bound are tight in general due to the Shannon
multigraphs, which are triangles whose multi-edges have balanced
multiplicities. (Note however that the latter bound can be improved for
specific choices of $\Delta(G)$ and $\mu(G)$, as described in the
work of Scheide and Stiebitz~\cite{ScSt09}.)

A notable conjecture on edge-colouring arose in the 1970s, on both sides of
the iron curtain. The Goldberg--Seymour Conjecture, due independently to
Goldberg~\cite{Gol73} and Seymour~\cite{Sey79}, asserts that
$\chi'(G)\in\{\,\Delta(G),\Delta(G)+1,\lceil\rho(G)\rceil\,\}$ for any
multigraph~$G$, where
\[\rho(G)=\max\Bigl\{\,\frac{2|E(G[T])|}{|T|-1}\Bigm|
T\subseteq V,\;|T|\ge3,\;\text{$|T|$ odd}\,\Bigr\}.\]
The parameter $\rho(G)$ is a lower bound on $\chi'(G)$ based on the maximum
ratio between the number of edges in~$H$ and the number of edges in a
maximum matching of~$H$, taken over induced subgraphs~$H$ of~$G$. This
conjecture remains open and is regarded as one of the most important
problems in chromatic graph theory. Perhaps the most outstanding progress
on this problem is due to Kahn~\cite{Kah96a}, who established an asymptotic
form.

The list variant of edge-colouring can be traced as far back as list
colouring itself. The concept of list colouring was devised independently
by Vizing~\cite{Viz76} and Erd\H{o}s, Rubin and Taylor~\cite{ERT80}, with
the iron curtain playing its customary role here too. The List Colouring
Conjecture (LCC) was already formulated by Vizing as early as 1975 and was
independently reformulated several times, a brief historical account of
which is given by, e.g., H{\"a}ggkvist and Janssen~\cite{HaJa97}. For more
on the LCC, particularly with respect to the probabilistic method, consult
the monograph of Molloy and Reed~\cite{MoRe02}. The results on the LCC most
relevant to our investigations also happen to be two of the most striking,
both from the mid-1990s. First, Galvin~\cite{Gal95} used a beautiful short
argument to prove Dinitz's Conjecture (concerning the extension of arrays to
partial Latin squares), which at the same time confirmed the LCC for
bipartite multigraphs. Not long after Galvin's work, Kahn applied powerful
probabilistic methods, with inspiration from extremal combinatorics and
statistical physics, to asymptotically affirm the LCC~\cite{Kah96a,Kah00}.
For more background on Kahn's proof, related methods, and improvements,
consult~\cite{HaJa97,MoRe00,MoRe02}.

Inspiration for this class of problems may also be taken from list
vertex-colouring. For instance, we utilise a degree-choosability criterion
due independently to Borodin~\cite{Bor77} and Erd\H{o}s, Rubin and
Taylor~\cite{ERT80}. See for example a survey of Alon~\cite{Alo93} for an
excellent (if older) survey on list colouring in somewhat more generality.
We should mention that part of the motivation for studying list colouring
was to use it to attack other, less constrained colouring problems. The
connection has gone back in the other direction as well, as precolouring
extension demonstrates.

Activity in the area of precolouring extension increased dramatically as a
result of the startling proof by Thomassen of planar
$5$-choosability~\cite{Tho94}; a key ingredient in that proof was a
particular type of precolouring extension from some pair of adjacent
vertices, according to a specific planar embedding. A little bit later,
Thomassen asked about precolouring extension for planar graphs under a more
general setup~\cite{Tho97}. Eliding the planarity condition,
Albertson~\cite{Alb98} quickly answered Thomassen's question and proved
more: in any $k$-colourable graph, for any set of vertices with pairwise
minimum distance at least~$4$, any precolouring of that set from the
palette $[k+1]$ can be extended to a proper colouring of the entire graph.
(This implies Proposition~\ref{prop:albertson1} above.)

Since Albertson's seminal work, a large body of research has developed
around precolouring extension. But this research has focused almost
exclusively on extension of vertex-colourings. One of the few papers we are
aware of that deals with edge-precolouring extension is by Marcotte and
Seymour~\cite{MS90}, in which a different type of necessary condition for
extension is studied --- curiously, this paper predates the above mentioned
activity in vertex-precolouring.

For planar graphs, there has been significant interest in both
edge-colouring and list edge-colouring. It is known that planar graphs $G$
with $\Delta(G)\ge7$ satisfy $\chi'(G)=\Delta(G)$. This was proved in 1965
by Vizing~\cite{Viz65} in the case $\Delta(G)\ge8$, and much later by
Sanders and Zhao~\cite{SaZh01} for $\Delta(G)=7$. We remark that
Theorem~\ref{thm:planar2} strengthens this for $\Delta(G)$ somewhat larger.
Vizing conjectured that the same can be said for planar graphs $G$ with
$\Delta(G)=6$, but this long-standing question remains open. Vizing also
noted that not every planar graph $G$ with $\Delta(G) \in \{4,5\}$ is
$\Delta(G)$-edge-colourable. Regarding list edge-colouring, Borodin,
Kostochka and Woodall~\cite{BKW97} proved the LCC for planar graphs with
maximum degree at least~$12$, i.e., they proved that such graphs have list
chromatic index equal to their maximum degree. The LCC remains open for
planar graphs with smaller maximum degree, though it is known that if
$\Delta(G)\le4$ or $\Delta(G)\ge8$, then $\ch'(G)\le\Delta(G)+1$ (Juvan,
Mohar and \v{S}krekovski~\cite{JMS99} for $\Delta(G)\le4$;
Bonamy~\cite{Bon15} for $\Delta(G)=8$; Borodin~\cite{Bor90} for
$\Delta(G)\ge9$). As noted above, it is not true that planar graphs $G$
with $\Delta(G)\in\{4,5\}$ are always $\Delta(G)$-edge-choosable.

\section{Necessity of the Distance-2 Condition}
\label{sec:counterexample}

In this section, we show that if we omit the distance-$2$ condition on the
precoloured matching then Conjecture~\ref{conj:main} becomes false whenever
$\Delta(G)\ge 4$. For each $t\ge3$, we construct a graph~$G_t$ of maximum
degree ${t+1}$ with the property that, using the palette $\MK=[t+2]$, there
is a matching~$M$ and a precolouring of~$M$ that cannot be extended to a
proper edge-colouring of all of~$G_t$.

Our construction is based on an observation by Anstee and
Griggs~\cite{anstee-griggs}. For $t\ge3$, let~$H_t$ be the graph obtained
from $K_{t,t}$ by subdividing one edge.

\begin{lemma}[Anstee and
  Griggs~\cite{anstee-griggs}]\label{lem:uncolorable}\mbox{}\\*
  For every $t\ge3$, the equality $\chi'(H_t)=\Delta(H_t)+1=t+1$ holds.
\end{lemma}

\begin{proof}
  Since~$H_t$ has ${2t+1}$ vertices, its largest matching has size~$t$.
  Since~$H_t$ has ${t^2+1}$ edges, we cannot cover all the edges with~$t$
  matchings.
\end{proof}

\noindent
Let $A,B\subseteq V(H_t)$ be the original partite sets of $K_{t,t}$, so
that~$A$ and~$B$ are independent sets of size~$t$ in~$H_t$, and the only
vertex of~$H_t$ not contained in $A\cup B$ is the vertex of degree~$2$.
Let~$H'_t$ be the graph obtained from~$H_t$ by attaching a pendant edge to
each vertex of~$H_t$, and for each $v\in V(H_t)$, let $v'$ be the other
endpoint of the pendant edge at~$v$. Finally, set
$M_0=\{\,vv'\mid v\in V(H_t)\,\}$. We precolour the matching $M_0$ by
colouring~$vv'$ colour~$1$ if $v\in A$, and colouring~$vv'$ colour~$2$
otherwise. Now we define the full graph~$G_t$ by taking $t+1$ disjoint
copies of~$H'_t$, and adding a new vertex~$v^*$ adjacent to the unique
vertex of degree~$3$ in each copy of~$H'_t$. The precoloured matching~$M$
in~$G_t$ is just the union of each precoloured matching~$M_0$ in each copy
of~$H'_t$, with the same precolouring. Figure~\ref{fig:counterexample}
shows~$G_3$.

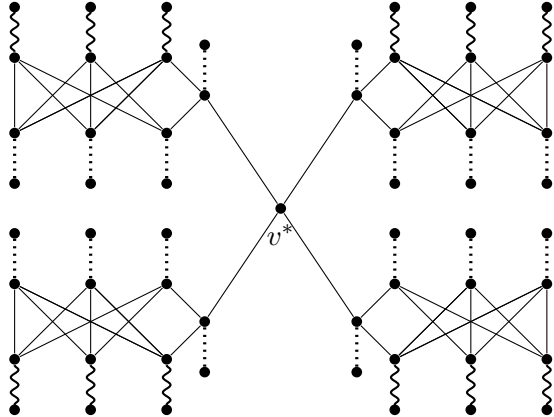
\begin{figure}
  \centering
  \begin{tikzpicture}
  \begin{scope}[xshift=-4.5cm, xscale=1]
    \foreach \i in {1,...,3}
    {
      \apoint{} (a1\i) at (\i cm, .5cm) {};
      \draw[thick, xred] (a1\i) -- ++(0cm, .67cm) node[style=vertex] {};
      \apoint{} (b1\i) at (\i cm, -.5cm) {};      
      \draw[thick, xblue] (b1\i) -- ++(0cm, -.67cm) node[style=vertex] {};
    }
    \foreach \i in {1,...,3}
    {
      \foreach \j in {1,2}
      {
        \draw (a1\i) -- (b1\j);
      }
    }
    \draw (a11) -- (b13); \draw (a12) -- (b13); \draw (b11) -- (a13); \draw (b12) -- (a13);
    \apoint{} (c1) at (3.5cm, 0cm) {}; \draw (c1) -- (a13); \draw (c1) -- (b13);    
    \draw[thick, xblue] (c1) -- ++(0cm, .67cm) node[style=vertex] {};
  \end{scope}
  \begin{scope}[xshift=4.5cm, xscale=-1]
    \foreach \i in {1,...,3}
    {
      \apoint{} (a2\i) at (\i cm, .5cm) {};
      \draw[thick, xred] (a2\i) -- ++(0cm, .67cm) node[style=vertex] {};
      \apoint{} (b2\i) at (\i cm, -.5cm) {};      
      \draw[thick, xblue] (b2\i) -- ++(0cm, -.67cm) node[style=vertex] {};
    }
    \foreach \i in {1,...,3}
    {
      \foreach \j in {1,2}
      {
        \draw (a2\i) -- (b2\j);
      }
    }
    \draw (a21) -- (b23); \draw (a22) -- (b23); \draw (b21) -- (a23); \draw (b22) -- (a23);
    \apoint{} (c2) at (3.5cm, 0cm) {}; \draw (c2) -- (a23); \draw (c2) -- (b23);    
    \draw[thick, xblue] (c2) -- ++(0cm, .67cm) node[style=vertex] {};
  \end{scope}
  \begin{scope}[xshift=4.5cm, yshift=-3cm, xscale=-1, yscale=-1]
    \foreach \i in {1,...,3}
    {
      \apoint{} (a3\i) at (\i cm, .5cm) {};
      \draw[thick, xred] (a3\i) -- ++(0cm, .67cm) node[style=vertex] {};
      \apoint{} (b3\i) at (\i cm, -.5cm) {};      
      \draw[thick, xblue] (b3\i) -- ++(0cm, -.67cm) node[style=vertex] {};
    }
    \foreach \i in {1,...,3}
    {
      \foreach \j in {1,2}
      {
        \draw (a3\i) -- (b3\j);
      }
    }
    \draw (a31) -- (b33); \draw (a32) -- (b33); \draw (b31) -- (a33); \draw (b32) -- (a33);
    \apoint{} (c3) at (3.5cm, 0cm) {}; \draw (c3) -- (a33); \draw (c3) -- (b33);    
    \draw[thick, xblue] (c3) -- ++(0cm, .67cm) node[style=vertex] {};
  \end{scope}
  \begin{scope}[xshift=-4.5cm, yshift=-3cm, xscale=1, yscale=-1]
    \foreach \i in {1,...,3}
    {
      \apoint{} (a4\i) at (\i cm, .5cm) {};
      \draw[thick, xred] (a4\i) -- ++(0cm, .67cm) node[style=vertex] {};
      \apoint{} (b4\i) at (\i cm, -.5cm) {};      
      \draw[thick, xblue] (b4\i) -- ++(0cm, -.67cm) node[style=vertex] {};
    }
    \foreach \i in {1,...,3}
    {
      \foreach \j in {1,2}
      {
        \draw (a4\i) -- (b4\j);
      }
    }
    \draw (a41) -- (b43); \draw (a42) -- (b43); \draw (b41) -- (a43); \draw (b42) -- (a43);
    \apoint{} (c4) at (3.5cm, 0cm) {}; \draw (c4) -- (a43); \draw (c4) -- (b43);    
    \draw[thick, xblue] (c4) -- ++(0cm, .67cm) node[style=vertex] {};
  \end{scope}
  \bpoint{v^*} (d) at (0cm, -1.5cm) {};
  \foreach \i in {1,...,4}
  {
    \draw (d) -- (c\i);
  }
\end{tikzpicture}
  \caption{A graph~$G_3$ with maximum degree~$4$ and a non-extendable
    precoloured matching using the palette~$[5]$. Wavy edges are
    precoloured~$1$, while dotted edges are precoloured~$2$.}
  \label{fig:counterexample}
\end{figure}

\begin{theorem}\label{thm:counterexample}\mbox{}\\*
  For every $t\ge 3$, using the palette $\MK=[t+2]=[\Delta(G_t)+\mu(G_t)]$,
  the precolouring of the matching~$M$ as described above cannot be
  extended to a proper edge-colouring of all of~$G_t$.
\end{theorem}

\begin{proof}
  Suppose to the contrary that~$G_t$ has an edge-colouring from $\MK$ that
  extends the precolouring of~$M$. Since every neighbour of~$v^*$
  has an incident edge precoloured~$2$, no edge incident to~$v^*$ can be
  coloured~$2$. Therefore, since $d(v^*) = t+1$, each of the ${t+1}$
  colours excluding~$2$ is used exactly once on the edges incident
  to~$v^*$. In particular, some edge~$e$ incident to~$v^*$ has colour~$1$.
  Let~$H$ be the copy of~$H_t$ containing the other endpoint of~$e$.
  Observe that no edge of~$H$ can be coloured~$1$ or~$2$: every edge
  joining~$A$ and~$B$ has an edge precoloured~$1$ at one endpoint and an
  edge precoloured~$2$ at the other, while the vertex of degree~$2$ in~$H$
  is incident to an edge precoloured~$2$ as well as the edge~$e$
  coloured~$1$. Hence all edges of~$H$ use only the~$t$ remaining colours.
  Since $\chi'(H_t) = t+1$ by Lemma~\ref{lem:uncolorable}, this is
  impossible.
\end{proof}

\section{Extensions of K\H{o}nig's and Shannon's Theorems}
\label{sec:bipartite}

Theorem~\ref{thm:predegree} below implies Theorem~\ref{thm:bipartite}, and
hence verifies Conjecture~\ref{conj:main} for bipartite multigraphs.
Theorem~\ref{thm:shannon} implies Theorem~\ref{thm:shannonmain}. Recall
that the precoloured degree of a vertex is the number of incident
precoloured edges.

\begin{theorem}\label{thm:predegree}\mbox{}\\*
  Let~$G$ be a bipartite multigraph and $k\ge1$. Take the palette as
  $\MK=[\Delta(G)+k]$. If~$G$ is properly precoloured so that the
  precoloured degree of any vertex is at most~$k$, then this precolouring
  can be extended to a proper edge-colouring of all of~$G$.
\end{theorem}

\begin{theorem}\label{thm:shannon}\mbox{}\\*
  Let~$G$ be a multigraph and $k\ge1$. Take the palette as
  $\MK=\bigl[\,\bigl\lfloor\tfrac32\Delta(G)+\tfrac12k\bigr\rfloor\,\bigr]$.
  If~$G$ is properly precoloured so that the precoloured degree of any
  vertex is at most~$k$, then this precolouring can be extended to a proper
  edge-colouring of all of~$G$.
\end{theorem}

\noindent
The two results are corollary to two theorems of Borodin, Kostochka and
Woodall~\cite{BKW97}. A multigraph~$G$ is \emph{$f$-edge-choosable}, where
$f:E(G)\to\mathbf{Z^+}$, if for any assignment of lists in which every
edge~$e$ receives a list of size at least $f(e)$, there is a proper
edge-colouring of~$G$ using colours from the lists.

\begin{theorem}[Borodin, Kostochka \& Woodall
  \cite{BKW97}]\label{thm:bkw}\mbox{}\\*
  Let~$G$ be a bipartite multigraph, and set $f(uv)=\max\{d(u),d(v)\}$ for
  each edge $uv\in E(G)$. Then~$G$ is $f$-edge-choosable.
\end{theorem}

\begin{theorem}[Borodin, Kostochka \&
  Woodall~\cite{BKW97}]\label{thm:bkwshannon}\mbox{}\\*
  Let~$G$ be a multigraph, and set
  $f(uv)=\max\{d(u),d(v)\}+
  \bigl\lfloor\tfrac12\min\{d(u),d(v)\}\bigr\rfloor$ for each edge
  $uv\in E(G)$. Then~$G$ is $f$-edge-choosable.
\end{theorem}

\noindent
Note that Theorem~\ref{thm:bkw} is a strengthening of Galvin's theorem;
while Theorem~\ref{thm:bkwshannon} is a list colouring version of Shannon's
theorem (and in fact follows from Theorem~\ref{thm:bkw}).

In our proofs of Theorems~\ref{thm:predegree} and~\ref{thm:shannon}, we use
the following refinement of Observation~\ref{cor:choose}. Given a graph $G$
and an edge $e$, the \emph{degree $d_G(e)$} of $e$ is the number of edges
adjacent to~$e$ in~$G$.

\begin{observation}\label{prop:choose}\mbox{}\\*
  Let $G=(V,E)$ be a multigraph. For a positive integer~$K$, take the
  palette as $\MK=[K]$. Suppose~$G$ is properly precoloured, with
  $S\subseteq E$ the precoloured edges. Set $E'=E\setminus S$, $G'=(V,E')$
  and $G''=(V,S)$. Suppose that~$G'$ is $f$-edge-choosable, for some
  function $f:E'\to\mathbf{Z^+}$. If for all $e\in E'$ we have
  $K-d_{G''+e}(e)\ge f(e)$, then the precolouring can be extended to a
  proper edge-colouring of all of~$G$.
\end{observation}

\begin{proof}[Proof of Theorems~\ref{thm:predegree} and~\ref{thm:shannon}]
  Assume that $G=(V,E)$ and $S\subseteq E$ is the set of precoloured edges.
  Set $E'=E\setminus S$, $G'=(V,E')$ and $G''=(V,S)$. Consider any
  uncoloured edge $e=uv\in E'$, and assume that $d_{G'}(u)\le d_{G'}(v)$.

  In the bipartite case, since
  $\Delta(G)-d_{G''}(v)\ge d_G(v)-d_{G''}(v)=d_{G'}(v)$ and
  $k\ge d_{G''}(u)$, we infer that
  \[|\MK|-d_{G''+e}(e)=
  \Delta(G)+k-\bigl(d_{G''}(u)+d_{G''}(v)\bigr)\ge d_{G'}(v)=
  \max\{\,d_{G'}(u),d_{G'}(v)\,\}.\]
  Theorem~\ref{thm:predegree} follows by combining Observation~\ref{prop:choose} and
  Theorem~\ref{thm:bkw}.

  In the general case (Theorem~\ref{thm:shannon}) we obtain
  \begin{align*}
    |\MK|-d_{G''+e}(e)&=
    \bigl\lfloor\tfrac32\Delta(G)+\tfrac12k\bigr\rfloor-
    \bigl(d_{G''}(u)+d_{G''}(v)\bigr)\\
    &=(\Delta(G)-d_{G''}(v))+ \bigl\lfloor\tfrac12\Delta(G)+
    \tfrac12k-(d_{G''}(u)\bigr\rfloor\\[2pt]
    &\ge d_{G'}(v)+\bigl\lfloor\tfrac12\bigl(\Delta(G)-
    d_{G''}(u)\bigr)\bigr\rfloor\ge
    d_{G'}(v)+\bigl\lfloor\tfrac12d_{G'}(u)\bigr\rfloor\\[2pt]
    &=\max\{\,d_{G'}(u),d_{G'}(v)\,\}+
    \bigl\lfloor\tfrac12\min\{\,d_{G'}(u),d_{G'}(v)\,\}\bigr\rfloor.
  \end{align*}
  This time combining Observation~\ref{prop:choose} with
  Theorem~\ref{thm:bkwshannon} completes the proof.
\end{proof}

\section{An Approach using Gallai Trees}
\label{sec:gallai}

In this section, we use a result due independently to Borodin~\cite{Bor77}
and to Erd\H{o}s, Rubin and Taylor~\cite{ERT80}. This is a list version of
an older result of Gallai~\cite{Gal63} on colour-critical graphs. A
connected graph all of whose blocks are either complete graphs or odd
cycles is called a \emph{Gallai tree}.

\begin{theorem}[Borodin~\cite{Bor77}, Erd\H{o}s, Rubin \&
  Taylor~\cite{ERT80}]\label{thm:gallaitree}\mbox{}\\*
  Given a connected graph $G=(V,E)$, let $\ell(v)$, for $v\in V$, be an
  assignment of lists where each vertex~$v$ receives at least $d(v)$
  colours. Then there is a proper colouring of~$G$ using colours from the
  lists, unless $G$ is a Gallai tree and $|\ell(v)|=d(v)$ for all~$v$.
\end{theorem}

\noindent
With this we prove Theorem~\ref{thm:gallai}, which implies
Theorem~\ref{thm:subcubic}.

\begin{proof}[Proof of Theorem~\ref{thm:gallai}]
  Assume to the contrary that the connected multigraph~$G$ and the
  non-negative integer~$k$ satisfy $\Delta(L(G))\le\Delta(G)+k$, but that,
  using the palette $\MK=[\Delta(G)+k]$, there is a proper
  edge-precolouring of~$G$ of the required type that does not extend to a
  proper edge-colouring of~$G$. For a vertex~$v$, let $K(v)\subseteq\MK$ be
  the set of colours appearing on the precoloured edges incident with~$v$,
  and set $k(v)=|K(v)|$.

  Let~$G'$ be obtained from~$G$ by deleting all precoloured edges. To each
  edge $e=uv$ in~$G'$, we assign a list $\ell(e)$ containing those colours
  in~$\MK$ not appearing on precoloured edges adjacent to~$e$ in~$G$. For
  any edge $e=uv$ in~$G'$ we obtain, using that
  $\Delta(G)+k\ge\Delta(L(G))\ge d_{L(G)}(e)$,
  \begin{equation}\label{eq1}
    \begin{array}{@{}r@{}l@{}}
      |\ell(e)|&{}=|\MK|-|K(u)\cup K(v)|=
      (\Delta(G)+k)-|K(u)\cup K(v)|\\[1mm]
      &{}\ge d_{L(G)}(e)-|K(u)\cup K(v)|\ge d_{L(G')}(e).
    \end{array}
  \end{equation}
  Since there is no extension of the precolouring of $L(G)$ to a full
  colouring of $L(G)$, it follows that $L(G')$ is not vertex-choosable with
  the lists $\ell(e)$, for $e\in E(G')$. In particular, there is a
  component~$C'$ of~$G'$ such that $L(C')$ is not vertex-choosable with the
  lists $\ell(e)$, for $e\in E(C')$. By Theorem~\ref{thm:gallaitree},
  $L(C')$ must be a Gallai tree such that $|\ell(e)|=d_{L(C')}(e)$ for
  every~$e$. This also means that we must have equality in all inequalities
  used to derive~\eqref{eq1}; in particular:
  \begin{subequations}
    \begin{align}
      &\text{for all $e\in E(C')$:}&&\hspace{-10mm}
      d_{L(G)}(e)=\Delta(L(G))=\Delta(G)+k;\label{eq2a}\\
      &\text{for all $e=uv\in E(C')$:}&&\hspace{-10mm}
      |\ell(e)|=d_{L(C')}(e)=d_{L(G)}(e)-|K(u)\cup K(v)|.\label{eq2b}
    \end{align}
  \end{subequations}
  Now note that $d_{C'}(v)+k(v)=d_G(v)\le\Delta(G)$ for every vertex~$v$.
  So, analogously to~\eqref{eq1} above, we infer that for each edge $e=uv$
  in~$C'$ the order of $\ell(e)$ is at least the degree in~$C'$ of each of
  its end-vertices:
  \begin{equation}\label{eq3}
    \begin{array}{@{}r@{}l@{}}
      |\ell(e)|&{}=(\Delta(G)+k)-|K(u)\cup K(v)|=
      (\Delta(G)+k)-k(u)-k(v)+|K(u)\cap K(v)|\\[1mm]
      &{}\ge d_{C'}(v)+(k-k(u))+|K(u)\cap K(v)|\ge d_{C'}(v).
    \end{array}
  \end{equation}

  We require the following statements.

  \begin{claim}\label{cl1}
    \ Every vertex in~$C'$ has at least two neighbours.
  \end{claim}

  \begin{proof}
    Suppose to the contrary that the vertex~$u$ has the vertex~$v$ as its
    unique neighbour. Then for the edge $e=uv$ we have
    $d_{L(C')}(e)=d_{C'}(v)-1$ (this holds even if~$uv$ is a multi-edge).
    But since $|\ell(e)|=d_{L(C')}(e)$, this gives $|\ell(e)|<d_{C'}(v)$,
    contradicting~\eqref{eq3}.
  \end{proof}

  \begin{claim}\label{cl2}
    \ If~$C'$ is a simple odd cycle, then $k=0$ and $G=C'$.
  \end{claim}

  \begin{proof}
    Suppose that~$C'$ is a simple odd cycle. If $e=uv$ is an edge in~$C'$,
    then $|\ell(e)|=d_{L(C')}(e)=d_{C'}(v)=2$. From this we can assume, by
    permuting the colours, that $\ell(e)=\{1,2\}$ for every $e\in L(C')$.
    (Indeed, the only way to assign lists of length~$2$ to the edges of an
    odd cycle in such a way that there is no proper colouring of the cycle
    using colours from the lists is by making all lists identical.) There
    must also be equality everywhere in~\eqref{eq3}. Combining that
    with~\eqref{eq2b} means in particular that for every edge $e=uv$ we
    have $K(u)\cup K(v)=\{\,3,4,\dotsc,\Delta(G)+k\,\}$ and
    $K(u)\cap K(v)=\varnothing$. By an easy parity argument, we can see
    that this is only possible if all the sets $K(u)$, $u\in V(C')$, are
    empty. This means that $k=0$ (and $\Delta(G)=2$). Since~$G$ is
    connected, if there are no precoloured edges, then~$G$ can have only
    one component, which must be~$C'$.
  \end{proof}

  \noindent
  We continue by considering the case that~$C'$ is not an odd cycle. Since
  line graphs are claw-free, it follows that odd cycle blocks of length at
  least five are impossible in $L(C')$. We deduce that all blocks of
  $L(C')$ are cliques. The only way that a leaf block~$B$ of $L(C')$ could
  be part of a nontrivial block structure is if it corresponds to a set of
  edges in~$C'$ that are all incident with a unique vertex, with one of the
  edges corresponding to the cut-vertex of~$B$. This is ruled out by
  Claim~\ref{cl1}. We conclude that $L(C')$ must itself be a clique. In
  turn, the only way that a line graph $L(C')$ of a multigraph is a clique
  is if~$C'$ is a star or a triangle, with possibly multiple edges. The
  first option is ruled out by Claim~\ref{cl1}, so~$C'$ must be a triangle,
  possibly with multi-edges.

  Let $u,v,w$ be the vertices in~$C'$ and set $m=|E(C')|$. Then for all
  $e\in E(C')$ we have $|\ell(e)|=d_{L(C')}(e)=m-1$. It is easy to check
  that with lists of this size, the only way that~$C'$ is not
  edge-choosable is if all the lists are the same. This also means that the
  sets $K(u)\cup K(v)$, $K(u)\cup K(w)$ and $K(v)\cup K(w)$ are the same.

  Let~$A(u)$ be the set of colours that appear on precoloured edges
  incident with~$u$, but not with~$v$ or~$w$; define $A(v)$ and $A(w)$
  analogously. (In other words, these are colours on the edges that
  connect~$C'$ to the rest of the graph~$G$.) Let~$D$ be the set of colours
  that appear on precoloured edges with end-vertices contained in
  $\{u,v,w\}$. From~\eqref{eq2a} and~\eqref{eq2b} we deduce that
  $|\ell(e)|=d_{L(C')}(e)$ for every edge $e$ in~$C'$, which, applied to an
  edge between~$u$ and~$v$, implies that $A(u)\cap A(v)=\varnothing$,
  $A(u)\cap D=\varnothing$ and $A(v)\cap D=\varnothing$. By symmetry,
  $A(v)\cap A(w)=\varnothing$, $A(u)\cap A(w)=\varnothing$ and
  $A(w)\cap D=\varnothing$.

  Now recall that all edges in~$C'$ must have the same list. Consequently,
  the disjointness of the sets~$A(u)$, $A(v)$ and~$A(w)$ implies that these
  three sets are empty. Thus we find that there are no precoloured edges
  between any of $u,v,w$ and the rest of the graph. Since~$G$ is connected,
  it follows that $V(G)=\{u,v,w\}$. Let $m(uv),m(uw),m(vw)$ be the
  multiplicities of the edges of~$G$. Then
  $\Delta(L(G))=m(uv)+m(uw)+m(vw)-1$, while
  $\Delta(G)=m(uv)+m(uw)+m(vw)-\min\{\,m(uv),m(uw),m(vw)\,\}$. Since
  $\Delta(L(G))=\Delta(G)+k$, we have shown that part~(b) of the statement
  of the theorem holds, completing the proof.
\end{proof}

\section{Planar Graphs}
\label{sec:planar}

In this section, for brevity we usually write~$\Delta$ for $\Delta(G)$.

In the next subsection we prove Conjecture~\ref{conj:main} for planar
graphs of large enough maximum degree (at least~$17$), which is the
assertion of Theorem~\ref{thm:planar}. As mentioned earlier, the LCC is
known to hold for planar graphs with maximum degree at least~$12$. This is
yet another result of Borodin, Kostochka and Woodall~\cite{BKW97}: they
indeed show that $\ch'(G)\le\Delta$ for such graphs~$G$. Combining this
with Observation~\ref{cor:choose} gives the bounds in lines~3 and~7 of
Table~\ref{tab:planar}. Since the former bound will be useful for us later
on, let us state it formally.

\begin{proposition}\label{obs:delta+2}\mbox{}\\*
  Let $G$ be a planar graph with maximum degree $\Delta(G)\ge12$. Using the
  palette $\MK=[\Delta(G)+2]$, any precoloured matching can be extended to
  a proper edge-colouring of all of~$G$.
\end{proposition}

\noindent
Borodin~\cite{Bor90} showed that $\ch'(G)\le\Delta+1$ for planar graphs~$G$
of maximum degree $\Delta\ge9$. Recently, Bonamy~\cite{Bon15} extended this
last statement to the case $\Delta=8$. Combining this result with
Observation~\ref{cor:choose} implies that for planar graphs with maximum
degree $\Delta\ge8$ a precoloured matching can be extended to a proper
colouring of the entire graph with the palette $[\Delta+3]$, while a
precoloured distance-$2$ matching can be extended with the palette
$[\Delta+2]$.

For smaller values of~$\Delta$, we can use Theorems~\ref{thm:shannonmain}
and~\ref{thm:subcubic}, and the result of Juvan, Mohar and
\v{S}krekovski~\cite{JMS99} that $\ch'(G)\le\Delta(G)+1$ for a planar
graph~$G$ with $\Delta(G)\le4$, to achieve several of the bounds in
Table~\ref{tab:planar}. In particular, it follows that $\Delta+4$ colours
suffice for any planar graph with maximum degree~$\Delta$.

The final proof we present is of Theorem~\ref{thm:planar2}. As discussed in
Subsection~\ref{sub:background}, Vizing conjectured~\cite{Viz65} that any
planar graph with maximum degree $\Delta\ge6$ has a
$\Delta$-edge-colouring. The examples in Figure~\ref{fig:extree} show that
this statement is false if we allow an adversarial precolouring of a
distance-$2$ matching. But does it remains true with the adversarial
precolouring of any distance-$3$ matching? We prove that this is indeed the
case if $\Delta\ge23$. We expect that this lower bound on~$\Delta$ can be
reduced, though, as noted before, certainly not below~$6$.

\medskip
The proofs of Theorems~\ref{thm:planar} and~\ref{thm:planar2} can be found
in the next two subsections. They use a common framework, terminology and
notation, which we outline now. Note that both adapt a nice trick of Cohen
and Havet~\cite{CoHa10}, which shortens the argument considerably.

Whenever considering a planar graph~$G$, we fix a drawing of~$G$ in the
plane. (So we really should talk about a \emph{plane graph}.) Because of
this fixed embedding we can talk about the \emph{faces} of the graph.
If~$G$ is connected, then the boundary of any face~$f$ forms a closed
walk~$W_f$.

We adopt the following notation to classify the vertices of a graph~$G$
according to their degree and their incidence with vertices of degree~$1$. Let~$V_i$ be
the set of vertices of degree~$i$. Also, identify by $T_i\subseteq V_i$ the
set of those vertices of degree~$i$ that are adjacent to a vertex of
degree~$1$, and set $U_i=V_i\setminus T_i$. Write $T=\bigcup_{i\ge1}\!T_i$
and $U=V(G)\setminus T$. We also adopt the shorthand notation $V_{[i,j]}$,
$U_{[i,j]}$ and $T_{[i,j]}$ to mean, respectively, the sets of vertices
in~$V$, $U$ and~$T$ with degrees between~$i$ and~$j$ inclusively.

\subsection{Proof of Theorem~\ref{thm:planar}}

If~$G$ is not connected, then we extend the edge-colouring one component at
a time. The colouring of a component~$C$ with $\Delta(C)\le16$ can be
extended using the results on lines 1\,--\,3 of Table~\ref{tab:planar}.
Next, the statement of Theorem~\ref{thm:planar} is true for graphs with
maximum degree~$17$ and exactly~$17$ edges. We use induction on $E(G)$, and
proceed with the induction step. So we may assume that~$G$ is connected and
has at least~$18$ vertices, since $\Delta\ge17$. Let~$M$ be a precoloured
matching.

We first observe that
\begin{equation}\label{eq:sum1}
  \text{if $uv\in E(G)\setminus M$, then $d(u)+d(v)\ge\Delta+3$.}
\end{equation}
Indeed, suppose that the inequality does not hold for some edge
$uv\notin M$. Then, by induction if $\Delta(G-uv)\ge17$ and by
Proposition~\ref{obs:delta+2} if $\Delta(G-uv)=16$, there exists an
extension of~$M$ to a colouring of all $G-uv$ using the palette $\MK$.
Since at most~$\Delta$ colours are used on the edges adjacent to~$uv$, we
can easily extend the colouring further to~$uv$. It follows from this
observation that~$G$ has no vertices of degree~$2$, that every vertex with
degree~$1$ is incident with an edge in~$M$ and that any vertex has at most
one neighbour of degree~$1$. We will use these facts often without
reference in the remainder of the proof.

For a face~$f$, let $V^-(f)=V(f)\setminus V_1$, and denote by $W^-_f$ the
sequence of vertices on the boundary walk~$W_f$ after removing vertices
from~$V_1$. For a vertex~$v$, let $v_1,v_2,\dotsc,v_{d(v)}$ be the
neighbours of~$v$, listed in clockwise order according to the drawing
of~$G$. Write~$f_i$ for the face incident with~$v$ lying between the edges
$vv_i$ and $vv_{i+1}$ (taking addition modulo $d(v)$ in
$\{1,\dotsc,d(v)\}$).

If $v\in T$ has a (unique) neighbour in~$V_1$, then we always choose~$v_1$
to be this neighbour. In that case we have $f_{d(v)}=f_1$; we denote that
face by~$f_1$ again. Note that it is possible for other faces to be the
same as well (if~$v$ is a cut-vertex), but we will not identify those
multiple names of the same face. So, if $v\in U$, then the faces around~$v$
in consecutive order are $f_1,f_2,\dotsc,f_{d(v)}$; while, if $v\in T$,
then the faces around~$v$ are $f_1,f_2,\dotsc,f_{d(v)-1}$.

\begin{claim}\label{claim:vdeltav3}
  \ $|V_\Delta|>|V_3|$.
\end{claim}

\begin{proof}
  Consider the set~$F$ of edges in $E(G)\setminus M$ with one end-vertex
  in~$V_3$ and the other in~$V_\Delta$. The subgraph with vertex set
  $V_3\cup V_\Delta$ and edge set~$F$ is bipartite; we assert it is
  acyclic. For suppose there exists an (even) cycle~$C$ with
  $E(C)\subseteq F$. By induction if $\Delta(G-E(C))\ge17$ and by
  Proposition~\ref{obs:delta+2} if $\Delta(G-E(C))\in\{15,16\}$, we can
  extend the precolouring of $M$ to $G-E(C)$ using the palette~$\MK$. But
  then we can further extend this colouring to the edges in~$C$, since each
  edge in~$C$ is adjacent to only $\Delta-1$ coloured edges, and even
  cycles are $2$-edge-choosable.

  Since each vertex in~$V_3$ is incident with at least two edges in~$F$, we
  have $|V_\Delta|+|V_3|>|F|\ge2|V_3|$. The claim follows.
\end{proof}

\noindent
We use a discharging argument to continue the proof of the theorem. First,
let us assign to each vertex~$v$ a charge

\smallskip
\qitem{$\alpha1$:} $\alpha(v)=3d(v)-6$,

\smallskip\noindent
and to each face~$f$ a charge

\smallskip
\qitem{$\alpha2$:} $\alpha(f)=-6$.

\smallskip
For each vertex~$v$ we define~$\beta(v)$ as follows.

\smallskip
\qitem{$\beta1$:} If $v\in V_\Delta$, then $\beta(v)=-2$.

\smallskip
\qitem{$\beta2$:} If $v\in V_3$, then $\beta(v)=2$.

\smallskip
\qitem{$\beta3$:} In all other cases, $\beta(v)=0$.

\smallskip
For each edge~$e=vu$, we define~$\gamma_e(v)$ and $\gamma_e(u)$ as follows.

\smallskip
\qitem{$\gamma1$:} If $v\in V_1$, then $\gamma_e(v)=-\gamma_e(u)=3$.

\smallskip
\qitem{$\gamma2$:} If $v,u\notin V_1$, then $\gamma_e(v)=\gamma_e(u)=0$.

\smallskip
Finally, for each face~$f$ and vertex~$v\in W^-_f$ we define~$\delta_f(v)$
and~$\delta_v(f)$ as follows.

\smallskip
\qitem{$\delta1$:} If $v\in T_3$, then $\delta_v(f)=-\delta_f(v)=1$.

\smallskip
\qitem{$\delta2$:} If $v\in U_3$, then $\delta_v(f)=-\delta_f(v)=\tfrac53$.

\smallskip
\qitem{$\delta3$:} If $v\in T$ and $4\le d(v)\le\Delta-2$, then
$\delta_v(f)=-\delta_f(v)=3-\dfrac{6}{d(v)-1}$.

\smallskip
\qitem{$\delta4$:} If $v\in U$ and $4\le d(v)\le\Delta-2$, then
$\delta_v(f)=-\delta_f(v)=3-\dfrac{6}{d(v)}$.

\smallskip
\qitem{$\delta5$:} If $d(v)\ge\Delta-1$, $|V^-(f)|=3$, and both neighbours
of~$v$ in $V^-(f)$ are vertices in $U_{[3,8]}$ that are joined by an edge
in~$M$, then $\delta_v(f)=-\delta_f(v)=3$.

\smallskip
\qitem{$\delta6$:} If $d(v)\ge\Delta-1$, $|V^-(f)|=3$, and~$v$ has a
neighbour in $V^-(f)\cap T_{[3,6]}$, then
$\delta_v(f)=-\delta_f(v)=\tfrac52$.

\smallskip
\qitem{$\delta7$:} If $d(v)\ge\Delta-1$, $|V^-(f)|=3$, none of~$\delta5$
and $\delta6$ applies, and~$v$ has a neighbour in $V^-(f)\cap U_{[3,5]}$,
then $\delta_v(f)=-\delta_f(v)=\tfrac94$.

\smallskip
\qitem{$\delta8$:} If $d(v)\ge\Delta-1$, $|V^-(f)|=3$, and none
of~$\delta5$, $\delta6$ and~$\delta7$ applies, then
$\delta_v(f)=-\delta_f(v)=2$.

\smallskip
\qitem{$\delta9$:} If $d(v)\ge\Delta-1$, $|V^-(f)|\ge4$, and~$v$ has a
neighbour in $V^-(f)\cap{T_{[3,6]}}$, then $\delta_v(f)=-\delta_f(v)=2$.

\smallskip
\qitem{$\delta10$:} If $d(v)\ge\Delta-1$, $|V^-(f)|\ge4$, and~$\delta9$
does not apply, then $\delta_v(f)=-\delta_f(v)=\tfrac32$.

\smallskip
For a vertex~$v$, write $\gamma(v)$ for the sum of $\gamma_e(v)$ over all
edges~$e$ that have~$v$ as an end-vertex. For a vertex~$v$ of degree~$1$ we
set $\delta(v)=0$. For every other vertex~$v$, write $\delta(v)$ for the
sum over the faces~$f$ around~$v$ of $\delta_f(v)$. Similarly, for a
face~$f$, write $\delta(f)$ for the sum over the vertices~$v$ on the
reduced walk~$W^-_f$ around~$f$ of the values of $\delta_v(f)$.

By the definitions of~$\gamma$ and~$\delta$,
\[\textstyle\sum_v\gamma(v)+\sum_v\delta(v)+\sum_f\delta(f)=0.\]
It follows from Claim~\ref{claim:vdeltav3} that
\[\textstyle\sum_v\beta(v)<0.\]
Finally, from Euler's formula for simple plane graphs, we obtain
\[\textstyle\sum_v\alpha(v)+\sum_f\alpha(f)<0.\]
Thus, in order to reach a contradiction, it is enough to show that for
every vertex~$v$:
\begin{equation}\label{eq:discvertex}
  \alpha(v)+\beta(v)+\delta(v)+\gamma(v)\ge0,
\end{equation}
and that for every face~$f$:
\begin{equation}\label{eq:discregion}
  \alpha(f)+\delta(f)\ge0.
\end{equation}

Let~$f$ be a face. As~$G$ is simple, $|V^-(f)|\ge3$. Since $\alpha(f)=-6$,
to establish~\eqref{eq:discregion} it is enough to show that
$\delta(f)\ge6$. Let~$v$ be a vertex in $V^-(f)$ for which $\delta_v(f)$ is
minimum. If $\delta_v(f)\cdot|V^-(f)|\ge6$, then~\eqref{eq:discregion}
clearly holds, and so we only need to deal with cases
$\delta1$\,--\,$\delta4$. Also, if
$v\in T_{[7,\Delta-2]}\cup U_{[6,\Delta-2]}$, then~$\delta3$ and~$\delta4$
give $\delta_v(f)\ge2$, and hence again~\eqref{eq:discregion} is verified.

If $v\in T_{[3,4]}$, then $\delta_v(f)=1$ by~$\delta1$ or~$\delta3$ and,
by~\eqref{eq:sum1}, the neighbours~$u$ and~$w$ of~$v$ in $V^-(f)$ have
degree at least $\Delta-1$. If $|V^-(f)|=3$, then~$\delta6$ applies to
both~$u$ and~$w$, so $\delta_u(f)=\delta_w(f)=\tfrac52$. If $|V^-(f)|\ge4$,
then~$\delta9$ applies to both~$u$ and~$w$, so $\delta_u(f)=\delta_w(f)=2$,
while a fourth vertex~$z$ in $V^-(f)$ satisfies $\delta_z(f)\ge1$ by the
definition of~$v$. So~\eqref{eq:discregion} always follows.

If $v\in T_{[5,6]}$, then $\delta_v(f)\ge\tfrac32$, so we may assume that
$|V^-(f)|=3$ (as $\delta_v(f)\le\delta_u(f)$ whenever $u\in V^-(f)$).
Moreover, $v$ has neighbours $u,w$ in $V^-(f)$ with degree at least
$\Delta-3\ge9$ as $\Delta\ge12$. If $d(u)\ge\Delta-1$, then~$\delta6$ gives
$\delta_u(f)=\tfrac52$. If $d(u)\in\{\Delta-3,\Delta-2\}$, then~$\delta3$
and~$\delta4$ give $\delta_u(f)\ge\tfrac94$, as $\Delta\ge12$. Since
similar bounds hold for~$\delta_w(f)$, we deduce that~\eqref{eq:discregion}
holds.

We are left with the case where $v\in U_{[3,5]}$. By~$\delta2$
and~$\delta4$ we find that $\delta_v(f)\ge\tfrac32$, and hence we again
only have to consider the case where $|V^{-}(f)|=3$. Rules
$\delta3$\,--\,$\delta7$ ensure that any other vertex~$u$ in $V^-(f)$ with
$d(u)\ge9$ satisfies $\delta_u(f)\ge\tfrac94$. So we may suppose that there
is a vertex $u\in V^-(f)$ with $d(u)\le8$. Since $|V^-(f)|=3$, we must in
fact have $uv\in E(G)$. Moreover, as $\Delta\ge11$, we know
by~\eqref{eq:sum1} that the edge~$uv$ belongs to the matching~$M$. This
means that $u\in U_{[3,8]}$. Let~$w$ be the third vertex in $V^-(f)$. If
$v\in U_{[3,4]}$, then $d(w)\ge\Delta-1$ by~\eqref{eq:sum1}, and so
$\delta_w(f)=3$ by~$\delta5$, confirming~\eqref{eq:discregion}. As the
final case, assume that $v\in U_5$ and recall that
$\delta_u(f)\ge\delta_v(f)=\tfrac95$. Since also $d(w)\ge\Delta-2$, one of
$\delta3$\,--\,$\delta5$ applies to~$w$, yielding that
$\delta_w(f)\ge\tfrac{12}{5}$, as $\Delta\ge13$. So again $\delta(f)\ge6$,
confirming~\eqref{eq:discregion} for all faces.

\medskip
Now let~$v$ be a vertex. Recall the convention that if $v\in T$, then the
two consecutive faces incident with both~$v$ and its neighbour of
degree~$1$ are counted as one face, while all other faces are counted
separately.

If $d(v)=1$, then $\alpha(v)=-3$ and $\gamma(v)=3$. Since
$\beta(v)=\delta(v)=0$, we immediately obtain~\eqref{eq:discvertex}.

Recall that~$G$ has no vertices of degree~$2$. If $d(v)=3$, then
$\alpha(v)=3$, while $\beta(v)=2$ by~$\beta2$. If $v\in T_3$, then
$\gamma(v)=-3$ and~$\delta1$ implies that $\delta(v)=-2$. If $v\in U_3$,
then $\gamma(v)=0$ and~$\delta2$ implies that $\delta(v)=-5$. This
confirms~\eqref{eq:discvertex} if $d(v)=3$.

Next suppose that $4\le d(v)\le\Delta-2$. Recall that $\alpha(v)=3d(v)-6$,
and observe that $\beta(v)=0$. If $v\in T$, then $\gamma(v)=-3$
by~$\gamma1$. By~$\delta3$ we have
$\delta(v)=(d(v)-1)\cdot\Bigl(-3+\dfrac{6}{d(v)-1}\Bigr)=9-3d(v)$.
Similarly, if $v\in U$, then $\gamma(v)=0$, and~$\delta4$ implies that
$\delta(v)=6-3d(v)$. This proves~\eqref{eq:discvertex} for those
vertices~$v$.

Now suppose that $d(v)\ge\Delta-1$. As a next step towards
proving~\eqref{eq:discvertex}, we consider the average value
of~$\delta_f(v)$ over the faces incident with~$v$. For convenience, set
$d'(v)=d(v)-1$ if $v\in T$, and set $d'(v)=d(v)$ if $v\in U$.

\begin{claim}\label{claim:avg}
  \ If $d(v)\ge \Delta-1$, then
  \[\sum_{i=1}^{d'(v)}\delta_{f_i}(v)\ge -\tfrac{5}{2}\cdot d'(v).\]
\end{claim}

\begin{proof}
  To obtain the desired bound, we group some of the faces around~$v$ into disjoint consecutive triples based on how~$\delta5$ applies to them with respect to $v$.
  Let~$J$ be the set of indices $j\in\{1,2,\dotsc,d'(v)\}$ such that~$\delta5$ applies to~$f_j$ with respect to $v$.
  The definition of~$\delta5$ precludes the possibility that~$\delta5$ applies to two consecutive faces around~$v$.
  Let~$K$ be any maximal set of indices $k\in\{1,2,\dotsc,d'(v)\}$ such that (modulo~$d'(v)$) both $k-1$ and $k+1$ are in $J$ and neither $k-2$ nor $k+2$ are in $K$.
  Let~$J_K = J \setminus \bigcup_{k\in K}\{k-1,k+1\}$.
  To define the triples, each face with index in $K\cup J_K$ is grouped with the two faces neighbouring it around~$v$.
  Note that by the maximality of $K$ these triples are all pairwise disjoint.
  If a face $f_i$ around~$v$ is not in a triple, then by $\delta6$\,--\,$\delta10$ we know that $\delta_{f_i}(v)\ge-\tfrac{5}{2}$. It follows that
  \[\sum_{i=1}^{d'(v)}\delta_{f_i}(v)\ge \sum_{j\in K\cup J_K}\bigl(\delta_{f_{j-1}}(v)+\delta_{f_{j}}(v)+\delta_{f_{j+1}}(v)\bigr)-\tfrac{5}{2}\bigl(d'(v)-3|K\cup J_K|\bigr),\]
  where the computation of indices is modulo~$d'(v)$ in $\{1,\dotsc,d'(v)\}$.
  Observe that for every index $j\in K$ only~$\delta10$ may apply to $f_j$ with respect to $v$ (by using~\eqref{eq:sum1} together with the assumption on $\Delta$ to exclude $\delta6$\,--\,$\delta8$, as well as a brief inspection of~$\delta9$), meaning that~$\delta_{f_j}(v)$ is~$-\tfrac{3}{2}$.
  Moreover, for every index $j\in J_K$ we see that~$\delta_{f_{j-1}}(v)$ and~$\delta_{f_{j+1}}(v)$ are both at least~$-\tfrac{9}{4}$, since~$\delta_6$ does not apply to~$f_{j-1}$ or~$f_{j+1}$ with respect to $v$ (and the same indeed is also the case for~$\delta_7$).
  We conclude for every index $j\in K\cup J_K$ that
  \[\delta_{f_{j-1}}(v)+\delta_{f_{j}}(v)+\delta_{f_{j+1}}(v)\ge-\tfrac{15}{2}.\]
  Hence we have in total
  \[\sum_{i=1}^{d'(v)}\delta_{f_i}(v)\ge
  -\tfrac{15}{2}|K\cup J_K|-\tfrac{5}{2}\bigl(d'(v)-3|K\cup J_K|\bigr)=
  -\tfrac{5}{2}d'(v).\qedhere\]
\end{proof}

\noindent
Claim~\ref{claim:avg} allows us to finish our analysis of the vertices.

First suppose that $d(v)=\Delta-1$. Then $\alpha(v)=3\Delta-9$ and
$\beta(v)=0$. If $v\in T$, then $\gamma(v)=-3$ and Claim~\ref{claim:avg}
gives $\delta(v)\ge-\tfrac{5}{2}(\Delta-2)$. Since $\Delta\ge14$,
inequality~\eqref{eq:discvertex} follows. If $v\in U$, then $\gamma(v)=0$
and $\delta(v)\ge-\tfrac{5}{2}(\Delta-1)$. The hypothesis that
$\Delta\ge13$ guarantees that~\eqref{eq:discvertex} is valid again.

Finally, suppose that $d(v)=\Delta$. Now $\alpha(v)=3\Delta-6$ and
$\beta(v)=-2$. If $v\in T$, then $\gamma(v)=-3$ and
$\delta(v)\ge-\tfrac{5}{2}\Delta+\tfrac{5}{2}$. We see
that~\eqref{eq:discvertex} holds, as $\Delta\ge17$. If $v\in U$, then
$\gamma(v)=0$ and $\delta(v)\ge-\tfrac{5}{2}\Delta$.
So~\eqref{eq:discvertex} is verified, provided that $\Delta\ge16$.

This confirms~\eqref{eq:discvertex} for all vertices
and completes the proof of the theorem.\qquad\hspace*{\fill}$\Box$

\subsection{Proof of Theorem~\ref{thm:planar2}}

Recall the notation and terminology given in the introduction of this
section.

Also this time, if~$G$ is not connected, then we extend the edge-colouring
one component at a time. The colouring of a component~$C$ with
$\Delta(C)\le22$ can be extended using the results on lines 1\,--\,3 of
Table~\ref{tab:planar}. Next, the statement of Theorem~\ref{thm:planar2} is
true for graphs with maximum degree~$23$ and exactly~$23$ edges. We use
induction on $E(G)$, and proceed with the induction step. So we may assume
that~$G$ is connected and has at least~$24$ vertices, since $\Delta\ge23$.
Let~$M$ be a precoloured distance-$3$ matching.

We first observe that
\begin{equation}\label{eq:sum3}
  \text{if $uv\in E(G)\setminus M$, then $d(u)+d(v)\ge\Delta+2$.}
\end{equation}
Indeed, suppose that the inequality does not hold for some $uv\notin M$.
Then by induction if $\Delta(G-uv)\ge23$ and by Theorem~\ref{thm:planar} if
$\Delta(G-uv)=22$, there exists an extension of~$M$ to a colouring of
$G-uv$ using the palette $\MK$. Since at most $\Delta-1$ colours are used
on the edges adjacent to~$uv$, we can easily extend the colouring further
to~$uv$. From~\eqref{eq:sum3} it follows that every vertex with degree~$1$
is incident with an edge in~$M$ and that if~$v$ has degree~$2$ and
$uv\notin M$, then $d(u)=\Delta$. In particular, if a vertex~$v$ with
degree greater than~$1$ has a neighbour in~$T_2$, then $d(v)=\Delta$.
Moreover, since edges in~$M$ are at distance at least~$4$ in~$G$, a vertex
can have at most one neighbour in $V_1\cup T_2$.

Let~$V_2'$ be the set of vertices of degree~$2$ that are not incident with
an edge of~$M$. For a face~$f$, let $V^-(f)=V(f)\setminus(V_1\cup T_2)$,
and let $W^-_f$ be the sequence of vertices on the boundary walk~$W_f$
after removing vertices from $V_1\cup T_2$. For a vertex~$v$, let
$v_1,v_2,\dotsc,v_{d(v)}$ be the neighbours of~$v$, listed in clockwise
order according to the drawing of~$G$. Write~$f_i$ for the face incident
with~$v$ lying between the edges $vv_i$ and $vv_{i+1}$ (taking addition
modulo $d(v)$ in $\{1,\dotsc,d(v)\}$).

If a vertex~$v$ has a (unique) neighbour in $V_1\cup T_2$, then we always
choose~$v_1$ to be this neighbour. In that case $f_{d(v)}=f_1$, and that
face is called~$f_1$ again. Note that it is possible for other faces to be
the same as well (if~$v$ is a cut-vertex), but we will not identify those
multiple names of the same face.

\begin{claim}\label{claim:vdeltav2}
  \ $|V_\Delta|>|V_2'|$.
\end{claim}

\begin{proof}
  Consider the set~$F$ of edges in~$E(G)$ with one end-vertex in~$V_2'$ and
  the other in~$V_\Delta$. Note that $F\cap M=\varnothing$ by the
  definition of~$V_2'$. The subgraph with vertex set $V_2'\cup V_\Delta$
  and edge set~$F$ is bipartite; we assert it is acyclic. For suppose there
  exists an (even) cycle~$C$ with $E(C)\subseteq F$. By induction if
  $\Delta(G-E(C))\ge23$, by Theorem~\ref{thm:planar} if
  $\Delta(G-E(C))=22$, and by Proposition~\ref{obs:delta+2} if
  $\Delta(G-E(C))=21$, we can extend the precolouring of~$M$ to $G-E(C)$
  using the palette~$\MK$. But then we can further extend this colouring to
  the edges of~$C$, since each one sees only $\Delta-2$ coloured edges, and
  even cycles are $2$-edge-choosable.

  Since each vertex in~$V_2'$ is incident with precisely two edges in~$F$,
  we have $|V_\Delta|+|V_2'|>|F|=2|V_2'|$. The claim follows.
\end{proof}

\noindent
We use a discharging argument to complete the proof. First, let us assign
to each vertex~$v$ a charge

\smallskip
\qitem{$\alpha1$:} $\alpha(v)=3d(v)-6$,

\smallskip\noindent
and to each face~$f$ a charge

\smallskip
\qitem{$\alpha2$:} $\alpha(f)=-6$.

\smallskip
For each vertex~$v$ we define~$\beta(v)$ as follows.

\smallskip
\qitem{$\beta1$:} If $v\in V_\Delta$, then $\beta(v)=-2$.

\smallskip
\qitem{$\beta2$:} If $v\in V_2'$, then $\beta(v)=2$.

\smallskip
\qitem{$\beta3$:} In all other cases, $\beta(v)=0$.

\smallskip
For each edge~$e=vu$, we define~$\gamma_e(v)$ and~$\gamma_e(u)$ as follows.

\smallskip
\qitem{$\gamma1$:} If $v\in V_1$, then $\gamma_e(v)=-\gamma_e(u)=3$.

\smallskip
\qitem{$\gamma2$:} If $v\in T_2$ and $u\in V_\Delta$, then
$\gamma_e(v)=-\gamma_e(u)=3$.

\smallskip
\qitem{$\gamma3$:} If $v\in U_2\setminus V_2'$ and $u\in V_\Delta$, then
$\gamma_e(v)=-\gamma_e(u)=2$.

\smallskip
\qitem{$\gamma4$:} In all other cases, $\gamma_e(v)=\gamma_e(u)=0$.

\smallskip
Finally, for each face~$f$ and vertex $v\in W^-_f$ we define $\delta_f(v)$
and $\delta_v(f)$ as follows.

\smallskip
\qitem{$\delta1$:} If $v\in U_2$, then $\delta_v(f)=-\delta_f(v)=1$.

\smallskip
\qitem{$\delta2$:} If $v\in T$ and $3\le d(v)\le\Delta-4$, then
$\delta_v(f)=-\delta_f(v)=3-\dfrac{6}{d(v)-1}$.

\smallskip
\qitem{$\delta3$:} If $v\in U$ and $3\le d(v)\le\Delta-4$, then
$\delta_v(f)=-\delta_f(v)=3-\dfrac{6}{d(v)}$.

\smallskip
\qitem{$\delta4$:} If $d(v)\ge\Delta-3$, $|V^-(f)|=3$, and both neighbours
of~$v$ in $V^-(f)$ are joined by an edge in~$M$, then
$\delta_v(f)=-\delta_f(v)=4$.

\smallskip
\qitem{$\delta5$:} If $d(v)\ge\Delta-3$ and~$v$ has a neighbour in
$V^-(f)\cap T_3$, then $\delta_v(f)=-\delta_f(v)=3$.

\smallskip
\qitem{$\delta6$:} If $d(v)\ge\Delta-3$ and none of~$\delta4$ and~$\delta5$
applies, then $\delta_v(f)=-\delta_f(v)=\tfrac52$.

\smallskip
For a vertex~$v$ and face~$f$, let $\gamma(v)$, $\delta(v)$ and $\delta(f)$
be defined as in the proof of Theorem~\ref{thm:planar}. By definition we
have $\sum_v\gamma(v)+\sum_v\delta(v)+\sum_f\delta(f)=0$. It follows from
Claim~\ref{claim:vdeltav2} that $\sum_v\beta(v)<0$. From Euler's formula we
obtain $\sum_v\alpha(v)+\sum_f\alpha(f)<0$.

Thus, in order to reach a contradiction, it is enough to show that for
every vertex~$v$:
\begin{equation}\label{eq:discvertex3}
  \alpha(v)+\beta(v)+\delta(v)+\gamma(v)\ge0,
\end{equation}
and that for every face~$f$:
\begin{equation}\label{eq:discregion3}
  \alpha(f)+\delta(f)\ge0.
\end{equation}

Let~$f$ be a face. As~$G$ is simple, $|V^-(f)|\ge3$. Since $\alpha(f)=-6$,
it follows that~\eqref{eq:discregion3} is verified if we can show that
$\delta(f)\ge6$. Let~$v$ be a vertex in $V^-(f)$ for which $\delta_v(f)$ is
minimum. If $\delta_v(f)\cdot|V^-(f)|\ge6$, then~\eqref{eq:discregion3}
clearly holds. So, by checking $\delta1$\,--\,$\delta6$, we see we only
have to consider the case where $v\in T_{[3,6]}\cup U_{[2,5]}$. (Recall
that vertices from $V_1\cup T_2$ do not appear in~$W^-_f$.)

If $v\in U_2$, then let~$u$ and~$w$ be the neighbours of~$v$. Consider
first the case where both~$u$ and~$w$ have degree~$\Delta$. Then they both
belong to $V^-(f)$, so~\eqref{eq:discregion3} follows, since
$\delta_v(f)=1$ and $\delta_u(f)\ge\tfrac52$, $\delta_w(f)\ge\tfrac52$ by
$\delta4$\,--\,$\delta6$. Suppose now that~$u$ has degree less
than~$\Delta$, which implies by~\eqref{eq:sum3} that $uv\in M$ and,
consequently, $vw\notin M$. In particular, $w\in V^-(f)$ and~$w$ has
degree~$\Delta$. Note also that necessarily $u\in V^-(f)$. If $|V^-(f)|=3$,
then $\delta_w(f)=4$ by~$\delta4$. As $\delta_u(f)\ge\delta_v(f)=1$, it
follows that~\eqref{eq:discregion3} holds. If $|V^-(f)|\ge4$, then~$u$ has
a neighbour~$u'$ in $V^-(f)\setminus\{v,w\}$. We assert that
$\delta_u(f)+\delta_{u'}(f)\ge\tfrac{5}{2}$. Indeed, because $uu'\notin M$,
we know by~\eqref{eq:sum3} and since $\Delta\ge 18$ that (at least) one
of~$u$ and~$u'$ has degree at least~$5$. Consequently, by
$\delta2$\,--\,$\delta6$ we know that
$\max\{\delta_f(u),\delta_f(u')\}\ge \frac{3}{2}$. Since
$\min\{\delta_f(u),\delta_f(u')\}\ge \delta_f(v)=1$ and
$\delta_f(w)\ge\frac{5}{2}$ by~$\delta5$\ and~$\delta6$, it follows
that~\eqref{eq:discregion3} holds.

If $v\in T_3$, then $\delta_v(f)=0$, but~$v$ has two neighbours in~$V^-(f)$
that have degree at least $\Delta-1$ each. Equation~\eqref{eq:discregion3}
then follows from~$\delta5$.

For the remaining cases we always have $\delta_v(f)\ge1$. Rules
$\delta2$\,--\,$\delta6$ ensure that any vertex $u\in V^-(f)$ with
$d(u)\ge13$ satisfies $\delta_u(f)\ge\tfrac52$; hence there can be at most
one such vertex and, in particular, a neighbour~$u$ of~$v$ in $V^-(f)$ must
have degree at most~$12$. As~$v$ itself has degree at most~$6$,
by~\eqref{eq:sum3} we have $uv\in M$, which also implies that
$\{u,v\}\subseteq U$. Hence in particular $v\in U_{[3,5]}$. Let~$w$ be the
neighbour of~$v$ in $V^-(f)\setminus\{u\}$. Since $v\in U_{[3,5]}$ and
$vw\notin M$, it necessarily holds that $d(w)\ge\Delta-3$. If $|V^-(f)|=3$,
then~\eqref{eq:discregion3} holds by~$\delta4$ since
$\delta_u(f)\ge\delta_v(f)\ge1$. If $|V^-(f)|\ge4$, then~$u$ has a
neighbour~$u'$ in $V^-(f)\setminus\{v,w\}$, which has degree at least
$\Delta+2-d(u)\ge10$. Consequently, $\delta_{u'}(f)\ge\tfrac{12}{5}$ by
$\delta3$, $\delta5$ or~$\delta6$. We deduce that~\eqref{eq:discregion3}
holds, as $\delta_w(f)\ge\tfrac{5}{2}$ by~$\delta5$ or~$\delta6$.

This confirms~\eqref{eq:discregion3} for all faces.

\medskip
Now let~$v$ be a vertex. Recall that $\alpha(v)=3d(v)-6$. Furthermore,
if~$v$ has a neighbour in $V_1\cup T_2$, then the two consecutive faces
incident with that neighbour are counted as one face; all other faces are
counted separately. Finally, as noted earlier, a vertex can have at most
one neighbour in $V_1\cup T_2$

If $d(v)=1$, then $\alpha(v)=-3$ and $\gamma(v)=3$. Since
$\beta(v)=\delta(v)=0$, we immediately obtain~\eqref{eq:discvertex3}.

If $d(v)=2$, then $\alpha(v)=0$. If $v\in T_2$, then both~$\gamma1$
and~$\gamma2$ apply; hence $\gamma(v)=0$. Again one can check that
$\beta(v)=\delta(v)=0$, confirming~\eqref{eq:discvertex3}. Otherwise
$v\in U_2$, and~$\delta1$ implies that $\delta(v)\ge-2$, as~$v$ is incident
with at most two faces. If $v\in V_2'$ as well, then~$\beta2$ yields that
$\beta(v)=2$ and $\gamma(v)=0$. If $v\notin V_2'$, then $\gamma(v)=2$ while
$\beta(v)=0$. In either case~\eqref{eq:discvertex3} follows.

Next suppose that $3\le d(v)\le\Delta-4$. Observe that $\beta(v)=0$. If
$v\in T$, then $\gamma(v)=-3$ by~$\gamma1$. Since~$v$ has a neighbour with
degree one, we know that~$v$ is incident with $d(v)-1$ regions, and
so~$\delta2$ yields that
$\delta(v)=(d(v)-1)\cdot\bigl({-3}+\dfrac{6}{d(v)-1}\bigr)= 9-3d(v)$.
Similarly, if $v\in U$, then $\gamma(v)=0$, and~$\delta3$ yields that
$\delta(v)=6-3d(v)$. This proves~\eqref{eq:discvertex3} for those
vertices~$v$.

Suppose now that $d(v)\in\{\,\Delta-3,\Delta-2,\Delta-1\,\}$. Then
$\beta(v)=0$. If $v\in T$, then $\gamma(v)=-3$ by~$\gamma1$. Since $M$ is
distance-$3$, none of~$\delta4$ and~$\delta5$ applies to~$v$, and~$v$ is
incident with $d(v)-1$ faces. From~$\delta6$ we deduce that
$\delta(v)=-\tfrac52(d(v)-1)$. Since $d(v)\ge\Delta-3\ge13$, it follows
that $3d(v)-6-3-\tfrac52(d(v)-1)=\tfrac12d(v)-\tfrac{13}2\ge0$, and
hence~\eqref{eq:discvertex3} is satisfied again. Next assume that $v\in U$,
and so $\gamma(v)=0$. The fact that~$M$ is distance-$3$ ensures
that~$\delta4$ applies to at most one face with respect to~$v$,
and~$\delta5$ applies to at most two faces with respect to~$v$.
Consequently, $\delta(v)\ge-\bigl(4+6+\tfrac52(d(v)-3)\bigr)$. Combined
with the assumption that $\Delta\ge17$, this is always enough to
satisfy~\eqref{eq:discvertex3}.

Finally, suppose that $d(v)=\Delta$. In this case $\beta(v)=-2$. If
$v\in T$, then the distance condition on~$M$ ensures that $\gamma(v)=-3$
and $\delta(v)=-\tfrac52(\Delta-1)$. Since $\Delta\ge17$ this
confirms~\eqref{eq:discvertex3}.

So we are left with the case where $v\in U$. Since~$M$ is a distance-$3$
matching, at most one of~$\gamma2$, $\gamma3$ applies and at most one of
$\delta4$,~$\delta5$ applies. Moreover, if~$\gamma2$ does apply, then
$\gamma(v)=-3$ and neither~$\delta4$ nor~$\delta5$ applies. This means that the
vertex~$v$ is incident with~$\Delta$ faces, and for each of those faces~$f$
we have $\delta_f(v)=-\frac52$. If~$\gamma2$ does not apply, then
$\gamma(v)\ge-2$. The vertex~$v$ is incident with~$\Delta$ faces, and for
$\Delta-1$ of those faces~$f$ we have $\delta_f(v)=-\frac52$. For the final
face~$f$ either $\delta4$ or~$\delta5$ may apply, so
$\delta_f(v)\in\{-4,-3,-\frac52\}$. Using that $\Delta\ge23$, we can check
that~\eqref{eq:discvertex3} is satisfied in all cases.

This confirms~\eqref{eq:discvertex3} for all vertices and completes the
proof of the theorem.\qquad\hspace*{\fill}$\Box$

\section{Avoiding Prescribed Colours on a Matching}
\label{sec:forbidden}

In this section, we show the following statement, which directly implies
Theorem~\ref{thm:forbidden}.

\begin{theorem}\label{thm:forbidden,stronger}\mbox{}\\*
  Let~$G$ be a multigraph with maximum degree $\Delta(G)$ and maximum
  multiplicity $\mu(G)$, and let~$M_1$ and~$M'$ be two disjoint
  matchings in~$G$. Suppose that each edge~$e$ of~$G$ is assigned a list
  $L(e)\subseteq[\Delta(G)+\mu(G)]$ of colours such that

  \smallskip
  \qitem{$\bullet$} $L(e)=\{1\}$ if $e\in M_1$;

  \smallskip
  \qitem{$\bullet$} $L(e)=\{2,\dots,\Delta(G)+\mu(G)\}$ if $e\in M'$; and

  \smallskip
  \qitem{$\bullet$} $L(e)=[\Delta(G)+\mu(G)]$ if
  $e\in E(G)\setminus (M_1\cup M')$.

  \smallskip\noindent
  Then there exists a proper edge-colouring~$\psi$ of~$G$ such that
  $\psi(e)\in L(e) $ for every $e\in E(G)$.
\end{theorem}

\noindent
To establish Theorem~\ref{thm:forbidden,stronger}, we use a result
mentioned just after Conjecture~\ref{conj:main}.

\begin{theorem}[Berge and Fournier~\cite{BeFo91}]\label{thm:BeFo91}\mbox{}\\*
  Let~$G$ be a multigraph with maximum degree $\Delta(G)$ and maximum
  multiplicity $\mu(G)$, and let~$M$ be a matching in~$G$. Then
  there exists a proper edge-colouring of~$G$ using the palette
  $[\Delta(G)+\mu(G)]$ such that every edge of~$M$ receives the same
  colour.
\end{theorem}

\begin{proof}[Proof of Theorem~\ref{thm:forbidden,stronger}]
  We may assume without loss of generality that $M_1$ is a maximal matching
  in $G\setminus M'$. We set
  \begin{equation*}
    B=\{\,e'\in M'\mid\text{$e'\cap e=\varnothing$ for all $e\in M_1$}\,\}.
  \end{equation*}
  Let~$\psi$ be a partial proper edge-colouring of~$G$ using colours in
  $[\Delta(G)+\mu(G)]$ such that{

  \smallskip
  \qitem{(i)} $\psi(e)=1$ for every $e\in M_1$;

  \smallskip
  \qitem{(ii)} $\psi(e')\ne1$ for every $e'\in M'$;

  \smallskip
  \qitem{(iii)} every edge of $E(G)\setminus B$ receives a colour under
  $\psi$; and

  \smallskip
  \qitem{(iv)} the number of edges of~$B$ that receive a colour under
  $\psi$ is maximal.}

  \smallskip
  To show that $\psi$ is well defined, we need to prove the existence of a
  partial proper edge-colouring of $G-B$ using the palette
  $[\Delta(G)+\mu(G)]$ that satisfies (i)\,--\,(iii).

  To this end, let $G' = G-B$. By Theorem~\ref{thm:BeFo91}, there is a
  proper edge-colouring $\phi$ of~$G'$ using colours in
  $[\Delta(G)+\mu(G)]$ such that every edge in~$M_1$ receives colour $1$.
  By the definition of~$B$, each edge in $M' \setminus B$ is incident to at
  least one edge in~$M_1$. Each edge in~$M_1$ receives colour $1$ under
  $\phi$ and therefore $\phi$ does not map any edge of $M'\setminus B$ to
  colour $1$. Thus $\phi$ ensures that $\psi$ exists.

  We now show that every edge of~$B$ receives a colour under $\psi$, which
  completes the proof. Suppose, on the contrary, that $xy \in B$ is an edge
  that is not coloured by~$\psi$. We start by making the following
  observations.

  \begin{claim}\label{clm:observation1}
    \ For every $e\in E(G)$, we have $\psi(e)=1$ if and only if $e\in M_1$.
  \end{claim}

  \noindent
  Indeed, if~$e$ is an edge that is coloured~$1$, then $e\notin M'$ and~$e$
  is not adjacent to an edge in~$M_1$, since all such edges are also
  coloured~$1$. Consequently, $e\in M_1$, as $M_1$ is a maximal
  matching of $G-M'$.

  Claim~\ref{clm:observation1} and the definition of~$B$ ensure the
  following.

  \begin{claim}\label{clm:observation2}
    \ Neither~$x$ nor~$y$ is incident with an edge that is coloured~$1$.
  \end{claim}

  \noindent
  For each vertex $v \in V(G)$ let $A_v\subseteq [\Delta(G)+\mu(G)]$ be the
  set of colours that do not appear on edges incident to~$v$.
  Claim~\ref{clm:observation2} states that $A_x$ and~$A_y$ both contain the
  colour $1$.

  \begin{claim}\label{clm:observation3}
    \ If $v\in N_G(x)\setminus \{y\}$, then $v$ is incident to an edge in
    $M_1$ and so $A_v$ does not contain the colour $1$.
  \end{claim}

  \noindent
  Indeed, for if~$v$ is not incident to an edge in~$M_1$, then by
  Claim~\ref{clm:observation2} the edge~$xv$ could be added to~$M_1$ to
  form a larger matching in $G-M'$, thereby contradicting the maximality
  of~$M_1$.

  We know that the edge~$xy$ is not yet coloured so both~$A_x$ and~$A_y$
  must contain some colour different from~$1$ and we shall from now on
  redefine~$A_y$ to be $A_y\setminus\{1\}$, which is not empty. We consider
  the following iterative procedure.

  Initially ($t=0$), we set $D_0=\{y\}$. At each step $t \ge 1$, we form
  the set~$D_t$ as follows:
  \begin{equation*}
    D_t=\textstyle\bigl\{\,v\in
    N_G(x)\setminus\bigcup\limits_{i=0}^{t-1}D_i\,\bigm| 
    \text{some edge between~$v$ and~$x$ has its colour in
      $\bigcup_{w\in D_{t-1}}\!A_w$}\,\bigr\}.
  \end{equation*}
  Since $\bigcup_{i\ge0}D_i\subseteq N_G(x)$ and $D_i\cap D_j=\varnothing$
  if $0\le i<j$, there exists a least non-negative integer~$t_0$ such that
  $D_{t_0+1}=\varnothing$. We define $D=\bigcup_{i\le t_0} D_i$. We
  consider now two cases.

  \medskip\noindent
  \textbf{Case 1.} \ Assume that there exist a vertex $w\in D$ and a colour
  $c\in A_w\cap A_x$. Since the subsets $D_0,\dotsc,D_{t_0}$ are pairwise
  disjoint, there is precisely one integer~$t_1$ such that $w\in D_{t_1}$.
  There exists a sequence $y=w_0,w_1,w_2,\dots,w_{t_1}=w$ of vertices such
  that $w_i\in D_i$ and (at least) one edge~$e_i$ between~$x$ and~$w_i$ has
  a colour in~$A_{w_{i-1}}$, whenever $1\le i\le t_1$.

  We may then define a partial proper edge-colouring~$\psi'$ of~$E(G)$,
  using colours in $[\Delta(G)+\mu(G)]$, with{

  \smallskip
  \qitem{$\bullet$} $\psi'(e)=\psi(e)$ if
  $e\notin\{\,e_i\mid1\le i\le t_1\,\}$;

  \smallskip
  \qitem{$\bullet$} $\psi'(e_i)=\psi(e_{i+1})$ for each
  $i\in\{0,\dotsc,t_1-1\}$; and

  \smallskip
  \qitem{$\bullet$} $\psi'(e_{t_1})=c$}.

  \smallskip
  One can check that $\psi'$ satisfies (i)\,--\,(iii) and colours one more
  edge of~$B$ than~$\psi$ does, which contradicts the choice of~$\psi$.

  \medskip
  For the second case, we need the following two observations.

  \begin{claim}\label{clm:observation4}
    \ For every~$z\in N_G(x)$, it holds that $\mu(G) \le |A_z|$.
  \end{claim}

  \noindent
  The only case that is not trivial is when $z=y$, due to our redefinition
  of~$A_y$. However, as the edge $xy$ is not coloured, the vertex~$y$ sees
  at most $\Delta(G)-1$ different colours, which implies the statement.

  Let~$H$ be the bipartite subgraph of~$G$ induced by the bipartition
  $(\{x\},D)$. (In particular, the edges of~$G$ between vertices in~$D$ are
  not in $H$.) The next statement follows directly from the fact that the
  number of coloured edges between~$x$ and~$y$ is at most $\mu(G)-1$.

  \begin{claim}\label{clm:observation5}
    \ The bipartite graph~$H$ contains fewer than $|D|\mu(G)$ coloured edges.
  \end{claim}

  \medskip\noindent
  We can now proceed with the second case.

  \medskip
  \noindent
  \textbf{Case 2.} \ For every vertex $w\in D$ and every colour $c\in A_w$,
  there exists an edge~$e_w$ between~$x$ and a vertex $z\in D$ such that
  $\psi(e_w)=c$. By Claims~\ref{clm:observation4}
  and~\ref{clm:observation5}, we know that the number of colours appearing
  in the bipartite graph~$H$ is less than $|D|\cdot\mu(G)$, which is at
  most $\sum_{w\in D} |A_w|$. This implies that there are two distinct
  vertices~$v_1$ and~$v_2$ in $D\subseteq N_G(x)$ with
  $A_{v_1}\cap A_{v_2}\ne\varnothing$. Let $c_1\in A_{v_1}\cap A_{v_2}$ and
  note that $c_1 \ne 1$ by Claim~\ref{clm:observation3}. Let
  $c_2 \in A_x \setminus\{1\}$. Then $c_2 \notin A_{v_1}\cup A_{v_2}$ and
  $c_1\notin A_x$. (And hence $c_1\ne c_2$.)

  For $i\in\{1,2\}$, let~$P_i$ be the maximal alternating path with colours
  $c_1$ and~$c_2$ beginning at~$v_i$. Note that~$x$ cannot belong to both
  paths. But if~$x$ does not belong to~$P_i$, then we may swap~$c_1$
  and~$c_2$ along the edges of~$P_i$. This leads us back to Case~1 because
  then $c_2$ belongs to $A_x\cap A_{v_i}$. (Note that such a swap affects
  neither the colours of the edges inside~$H$ nor those of edges in~$M_1$.)

  \medskip
  We have shown that in each case there exists a partial proper
  edge-colouring using colours in $[\Delta(G)+\mu(G)]$ and satisfying
  (i)\,--\,(iii) that assigns colours to more edges of~$B$ than $\psi$
  does, a contradiction.
\end{proof}

\section{Conclusion}
\label{sec:conclusion}

During the preparation of this manuscript, we learned of a related work in
the context of graph limits~\cite{CLP16}, in which is proposed the
following conjecture that has a similar flavour to our
Conjecture~\ref{conj:main}.

\begin{conjecture}[Cs\'oka, Lippner and
  Pikhurko~\cite{CLP16}]\label{conj:graphing}\mbox{}\\*
  Let~$G$ be a graph such that every vertex is of degree at most~$d$,
  except one of degree $d+1$. Using the palette $\MK=[d+1]$, suppose that
  at most $d-1$ pendant edges are precoloured. This precolouring can
  be extended to a proper edge-colouring of all of~$G$.
\end{conjecture}

\noindent
The authors of Conjecture~\ref{conj:graphing} proved the weaker statement
with $\MK=[d+9\sqrt{d}]$ instead of $\MK=[d+1]$.

\medskip
With respect to Question~\ref{question:main}, rather than imposing
conditions on the matching~$M$, we could instead constrain the
precolouring. In the light of Theorem~\ref{thm:counterexample} and the
result of Berge and Fournier~\cite{BeFo91}, the following is a natural
strengthened version of Conjecture~\ref{conj:main}.

\begin{conjecture}\label{conj:main2}\mbox{}\\*
  Let~$G$ be a multigraph with maximum degree $\Delta(G)$ and maximum
  multiplicity $\mu(G)$. Using the palette $\MK=[\Delta(G)+\mu(G)]$, any
  precoloured matching such that no two edges precoloured differently are
  within distance~$2$ can be extended to a proper edge-colouring of all
  of~$G$.
\end{conjecture}

\noindent
We may rephrase Theorem~\ref{thm:forbidden} in the language of list colouring
as follows: for any multigraph~$G$, any matching~$M$ in~$G$, and any list
assignment $L:E(G) \to 2^{[\Delta(G)+\mu(G)]}$ such that $|L(e)|=
\Delta(G)+\mu(G)-1$ if $e\in M$ and $L(e) = [\Delta(G)+\mu(G)]$ otherwise,
there is a proper $L$-edge-colouring of~$G$.  Theorem~\ref{thm:counterexample}
still leaves open the possibility that the following holds.

\begin{conjecture}\label{conj:main3}\mbox{}\\*
  Let~$G$ be a multigraph with maximum degree $\Delta(G)$ and maximum
  multiplicity $\mu(G)$ and let~$M$ be a matching in $G$. Let
  $L:E(G)\to2^{[\Delta(G)+\mu(G)]}$ be a list assignment such that
  $|L(e)|=2$ if $e\in M$ and $L(e) = [\Delta(G)+\mu(G)]$ otherwise. Then
  there is a proper $L$-edge-colouring of~$G$.
\end{conjecture}

\noindent
It would also be interesting if either of Conjectures~\ref{conj:main2}
and~\ref{conj:main3} could be confirmed with the constant~$2$ replaced by
any larger fixed integer.

\subsection*{Acknowledgement}

The authors would like to thank the anonymous referees for their meticulous reading
and for their corrections and suggestions, which improved the article significantly.

\bibliographystyle{abbrv}
\bibliography{edgeprecolour}

\end{document}